\documentclass[a4paper,11pt,reqno,twoside]{amsart}
\usepackage{GL_3_amplification}
\begin{document}%
\section{Introduction and statement of the results}%
\subsection{Motivation: sup-norms of $GL(n)$ Hecke-Maa$\ss$ cusp forms}%
Let $f$ be a $GL(n)$ Maa$\ss$ cusp form and $K$ be a fixed compact subset of $SL_n(\R)/SO_n(\R)$ (see \cite{MR2254662}). The generic or local bound for the sup-norm of $f$ restricted to $K$ is given by
\begin{equation*}
\abs{\abs{f\vert_{K}}}_\infty\ll\lambda_f^{n(n-1)/8}
\end{equation*}
where $\lambda_f$ is the Laplace eigenvalue of $f$ (see \cite{SaMo}). Note that F.~Brumley and N.~Templier noticed in \cite{BrTe} that the previous bound does not hold when $n\geq 6$ if $f$ is not restricted to a compact.
\par
If $f$ is an eigenform of the Hecke algebra, however, then the generic bound is not expected to be the correct order of magnitude for the sup-norm of the restriction of $f$ to a fixed compact.  This is essentially due to the fact that the Hecke operators are additional symmetries on the 
ambient space. In other words, we expect there to exist an absolute positive constant $\delta_n>0$ such that
\begin{equation}\label{eq_expect}
\abs{\abs{f\vert_{K}}}_\infty\ll\lambda_f^{n(n-1)/8-\delta_n}.
\end{equation}
\par
H.~Iwaniec and P.~Sarnak proved the bound given in \eqref{eq_expect} in \cite{MR1324136} when $n=2$ for $\delta_2=1/24$. Note that improving this constant $\delta_2$ seems to be a very delicate open problem. The case $n=3$ was completed by the authors in \cite{HoRiRo} with $\delta_3=1/124$. The general case was done, without an explicit value for $\delta_n$, in a series of recent impressive works by V.~Blomer and P.~Maga in \cite{BlMa1} and in \cite{BlMa2}. 
\par
All of the above achievements (and much more) were made possible thanks to generalizations of the amplification method developed by W.~Duke, J.~Friedlander and H.~Iwaniec for $GL(1)$ and $GL(2)$ (see \cite{MR1137896}, \cite{MR1166510} and \cite{MR1258904} for example).  In particular, the proof of \eqref{eq_expect} for $n=3$ with $\delta_3=1/124$ relies on Theorem \ref{theo_A} of this article which was stated without proof in \cite{HoRiRo} as Proposition 4.1\footnote{Theorem \ref{theo_A} and Proposition 4.1 in \cite{HoRiRo} are not entirely identical.  Since releasing our first article \cite{HoRiRo}, we have noticed a simplification in the construction of the amplifier.  Therefore, Theorem \ref{theo_A} only contains the identities in Proposition 4.1 of \cite{HoRiRo} which are necessary for the amplification method.  The implied power saving in the Laplace eigenvalue for the sup-norm bound remains the same.}. For the sake of completeness and future use, we provide the full details of the proof of Theorem \ref{theo_A}, including computations, here in this article. 
\subsection{The $GL(2)$ and $GL(3)$ amplifier}%
The general principle behind the construction of an amplifier, is the existence of an identity which allows one to write a non-zero constant as a finite sum of Hecke eigenvalues.  In the most basic context of $GL(2)$ automorphic forms, this identity is 
\begin{equation}\label{gl2id}
\lambda_f(p)^2-\lambda_f(p^2)=1 
\end{equation}
where $p$ is any prime and $\lambda_f(n)$ is the $n$-th Hecke eigenvalue of a Hecke-Maa\ss\;cusp form $f$ of full level, i.e.  $T_n f=\lambda_f(n) f$ where
$$
(T_n f)(z)=\frac{1}{\sqrt{n}} \sum_{ad=n}\; \sum_{b=1}^d f\left(\frac{az+b}{d}\right).
$$
One may interpret the above identity as the fact that the Rankin-Selberg convolution factors as the product of the adjoint square and the Riemann zeta function and therefore has a pole at $s=1$.  
\par
From the identity \eqref{gl2id}, one constructs an amplifier 
$$
A_f:=\left|\sum_\ell \alpha_\ell \lambda_f (\ell) \right|^2
$$
with 
\begin{equation*}
\alpha_{\ell}\coloneqq \begin{cases}
\lambda_{f_0}(\ell) & \text{if $\ell \leq \sqrt{L}$ is a prime number,} \\
-1 & \text{if $\ell \leq L$ is the square of a prime number,} \\
0 & \text{otherwise}
\end{cases}
\end{equation*}
for some fixed form $f_0$.  The advantage to constructing such an amplifier is that it is expected to be small\footnote{This follows, at least conditionally, from a suitable version of the Generalized Riemann Hypothesis. However, we do not need to use this fact.} for general forms $f$ while satisfying the lower bound $A_{f} \gg_\varepsilon L^{1-\varepsilon}$ when $f=f_0$.  
\par
Reinterpreting \eqref{gl2id} in terms of Hecke operators, we may write
$$
T_p \circ T_p - T_{p^2}=Id.
$$
In application to the sup-norm problem for GL(2) via a pre-trace formula argument, this translates into a need to geometrically understand the behavior of the following collection of operators on an automorphic kernel:
$$
T_p, T_p \circ T_q^\ast, T_p \circ T_{q^2}^\ast \textnormal{ and } T_{p^2} \circ T_{q^2}^\ast
$$
both in the cases of primes $p=q$ and $p\neq q$. Since the Hecke operators $T_n$ in $GL(2)$ are self-adjoint and computationally pleasant to work with due to their relatively simple composition law, one quickly computes that the above collection of Hecke operators may be reduced to the study of
$$
T_p, T_{pq}, T_{pq^2} \textnormal{ and } T_{p^2 q^2}.
$$
In truth, one has an opportunity to further reduce the collection of necessary Hecke operators through the simple inequality
\begin{equation}\label{ampineq}
A_f \leqslant 2 \left|\sum_{p} \alpha_p \lambda_f (p) \right|^2 + 2 \left|\sum_{p} \alpha_{p^2} \lambda_f (p^2) \right|^2.
\end{equation}
Indeed, an appropriate application of \eqref{ampineq} (see for example \cite{BHM}) allows one to further restrict the set of necessary Hecke operators to
$$
T_{pq} \textnormal{ and } T_{p^2 q^2}
$$
both in the cases of primes $p=q$ and $p\neq q$.
\par
The case of $GL(3)$ is much more computationally involved due to the lack of self-adjointness of the Hecke operators and their multiplication law.  Instead of looking at identities involving Hecke eigenvalues, we start immediately with the Hecke operators themselves (see \S \ref{sec_GL3} for definitions).  Our fundamental identity now will be
\begin{equation}\label{gl3id}
T_{\textnormal{diag}(1,p,p)}\circ T_{\textnormal{diag}(1,1,p)} - T_{\textnormal{diag}(1,p,p^2)} = (p^2+p+1) Id.
\end{equation}
We set $c_f(p)=a_f(p,1)$, $c_f(p)^\ast=\overline{a_f(p,1)}$ and $c_f(p^2)$ to be the eigenvalues of $p^{-1}T_{\textnormal{diag}(1,1,p)}=T_p$, $p^{-1}T_{\textnormal{diag}(1,p,p)}=T_p^\ast$ and $p^{-2}T_{\textnormal{diag}(1,p,p^2)}$ respectively when acting on a form $f$. See \eqref{eq_Hecke_def} and \eqref{eq_Hecke} for the precise definitions. We construct the amplifier
$$
A_f:=\left|\sum_\ell \alpha_\ell c_f (\ell) \right|^2
$$
with\footnote{One may also choose a variant, in the spirit of \cite{Ven}, which involves the signs of $c_{f_0}(\ell)$.}
\begin{equation*}
\alpha_{\ell}\coloneqq \begin{cases}
c_{f_0}(\ell)^\ast & \text{if $\ell \leq \sqrt{L}$ is a prime number,} \\
-1 & \text{if $\ell \leq L$ is the square of a prime number,} \\
0 & \text{otherwise}.
\end{cases}
\end{equation*}
As in the case of $GL(2)$, this amplifier will satisfy $A_{f_0} \gg_\varepsilon L^{1-\varepsilon}$ and $A_f$ is otherwise expected to be small for $f \neq f_0$. 
\par
Applying the inequality
\begin{equation}\label{ampineq3}
A_f \leqslant 2 \left|\sum_{p} \alpha_p c_f (p) \right|^2 + 2 \left|\sum_{p} \alpha_{p^2} c_f (p^2) \right|^2,
\end{equation}
one is reduced to understanding the actions of 
$$
T_{\textnormal{diag}(1,1,p)} \circ T_{\textnormal{diag}(1,q,q)} \textnormal{ and } T_{\textnormal{diag}(1,p,p^2)} \circ T_{\textnormal{diag}(1,q,q^2)} 
$$
both in the cases of primes $p=q$ and $p\neq q$ on the relevant automorphic kernel.  In the following sections, we compute the above compositions as linear combinations of other Hecke operators and state our main result as Theorem \ref{theo_A}.  In the end, we shall see that the following operators are the relevant ones for our application
$$
T_{\textnormal{diag}(1,p,pq)}, T_{\textnormal{diag}(1,pq,p^2q^2)}, T_{\textnormal{diag}(1,p^3,p^3)} \textnormal{ and } T_{\textnormal{diag}(1,1,p^3)}
$$
for primes $p=q$ and $p\neq q$.
\subsection{Statement of the results}%
\begin{theoint}\label{cor_A}
Let $p$ be a prime number and $\Lambda=GL_3(\Z)$.
\begin{itemize}
\item
The set $R_{1,1,p}$ (respectively $R_{1,p,p}$, $R_{1,p,p^2}$) defined in Proposition \ref{propo_dec_1_1_p} (respectively Proposition \ref{propo_dec_1_p_p}, Proposition \ref{propo_dec_1_p_p^2}) is a complete system of representatives for the distinct $\Lambda$-right cosets in the $\Lambda$-double coset of $\textnormal{diag}(1,1,p)$ (respectively $\textnormal{diag}(1,1,p)$, $\textnormal{diag}(1,p,p^2)$) modulo $\Lambda$.
\item
The following formulas for the degrees of $\Lambda$-double cosets hold.
\begin{eqnarray*}
\textnormal{deg}\left(\textnormal{diag}(1,1,p)\right) & = & p^2+p+1, \\
\textnormal{deg}\left(\textnormal{diag}(1,p,p)\right) & = & p^2+p+1, \\
\textnormal{deg}\left(\textnormal{diag}(1,p,p^2)\right) & = & p(p+1)(p^2+p+1), \\
\textnormal{deg}\left(\textnormal{diag}(p,p,p)\right) & = & 1, \\
\textnormal{deg}\left(\textnormal{diag}(1,p^2,p^4)\right) & = & p^5(p+1)(p^2+p+1), \\
\textnormal{deg}\left(\textnormal{diag}(1,p^3,p^3)\right) & = & p^4(p^2+p+1), \\
\textnormal{deg}\left(\textnormal{diag}(p,p,p^4)\right) & = & p^4(p^2+p+1), \\
\textnormal{deg}\left(\textnormal{diag}(p,p^2,p^3)\right) & = & p(p+1)(p^2+p+1), \\
\textnormal{deg}\left(\textnormal{diag}(p^2,p^2,p^2)\right) & = & 1.
\end{eqnarray*}
\item
Finally,
\begin{equation}\label{eq_intro_1}
\Lambda\textnormal{diag}(1,1,p)\Lambda\ast\Lambda\textnormal{diag}(1,p,p)\Lambda=\Lambda\textnormal{diag}(1,p,p^2)\Lambda+(p^2+p+1)\Lambda\textnormal{diag}(p,p,p)\Lambda.
\end{equation}
and
\begin{multline}\label{eq_intro_2}
\Lambda\textnormal{diag}(1,p,p^2)\Lambda\ast\Lambda\textnormal{diag}(1,p,p^2)\Lambda=\Lambda\textnormal{diag}(1,p^2,p^4)\Lambda+(p+1)\Lambda\textnormal{diag}(1,p^3,p^3)\Lambda \\
+(p+1)\Lambda\textnormal{diag}(p,p,p^4)\Lambda+(p+1)(2p-1)\Lambda\textnormal{diag}(p,p^2,p^3)\Lambda \\
+p(p+1)(p^2+p+1)\Lambda\textnormal{diag}(p^2,p^2,p^2)\Lambda.
\end{multline}
\end{itemize}
\end{theoint}
\begin{remark}
In \cite{MR0223304}, T.~Kodama explicitely computed the product of other double cosets in the slightly harder case of the Hecke ring for the symplectic group. The results stated in the previous theorem are similar in spirit.
\end{remark}
\begin{remark}
It is well-known that a $\Lambda$-double coset can be identified with its characteristic function $\chi$. Under this identification, the multiplication law between $\Lambda$-double cosets is the classical convolution between functions. If $\mu=(\mu_1,\mu_2,\mu_3)$ with $\mu_1\geq\mu_2\geq\mu_3\geq 0$ and $\nu=(\nu_1,\nu_2,\nu_3)$ with $\nu_1\geq\nu_2\geq\nu_3\geq 0$ are two partitions of length less than $n$, then
\begin{equation*}
\chi_{\Lambda\textnormal{diag}\left(p^{\mu_1},p^{\mu_2},p^{\mu_3}\right)\Lambda}\ast\chi_{\Lambda\textnormal{diag}\left(p^{\nu_1},p^{\nu_2},p^{\nu_3}\right)\Lambda}=\sum_{\lambda}g_{\mu,\nu}^\lambda(p)\chi_{\Lambda\textnormal{diag}\left(p^{\lambda_1},p^{\lambda_2},p^{\lambda_3}\right)\Lambda}
\end{equation*} 
for any prime number $p$ where $\lambda$ ranges over the partitions of length less than $n$ and the $g_{\mu,\nu}^\lambda(p)$ are the Hall polynomials (see \cite[Equation (2.6) Page 295]{MR1354144}). The sum on the right-hand side of the above equality is finite since only a finite number of the Hall polynomials are non-zero. However, determining which Hall polynomials vanish is not straightforward (see \cite[Equation (4.3) Page 188]{MR1354144}). Using Sage, one can check that
\begin{eqnarray*}
g_{(2,1,0),(2,1,0)}^{(4,2,0)}(p) & = & 1, \\
g_{(2,1,0),(2,1,0)}^{(3,3,0)}(p) & = & p+1, \\
g_{(2,1,0),(2,1,0)}^{(4,1,1)}(p) & = & p+1, \\
g_{(2,1,0),(2,1,0)}^{(3,2,1)}(p) & = & (p+1)(2p-1), \\
g_{(2,1,0),(2,1,0)}^{(2,2,2)}(p) & = & p(p+1)(p^2+p+1)
\end{eqnarray*}
and one can recover the coefficients occurring in \eqref{eq_intro_2}. We prefer to give a different proof, which has the advantage of producing explicit systems of representatives for the $\Lambda$-right cosets and formulas for the degrees.
\end{remark}
\begin{corint}\label{theo_A}
If $p$ and $q$ are two prime numbers then
\begin{equation}\label{eq_lin_2}
T_{\textnormal{diag}(1,p,p)} \circ T_{\textnormal{diag}(1,1,q)}  =  T_{\textnormal{diag}(1,p,pq)}+\delta_{p=q}(p^2+p+1)\text{Id}
\end{equation}
and
\begin{multline}\label{eq_lin_6}
T_{\textnormal{diag}(1,p,p^2)} \circ T_{\textnormal{diag}(1,q,q^2)} =T_{\textnormal{diag}(1,pq,p^2q^2)}
+\delta_{p=q}(p+1)\left(T_{\textnormal{diag}(1,p^3,p^3)}+T_{\textnormal{diag}(1,1,p^3)}\right) \\
+\delta_{p=q}(p+1)(2p-1)T_{\textnormal{diag}(1,p,p^2)}+\delta_{p=q}p(p+1)(p^2+p+1) Id.
\end{multline}
\end{corint}
When $p\neq q$, the previous corollary immediately follows from \eqref{eq_ex_multi} and \eqref{eq_ex_multi_2}. When $p=q$, it is a consequence of the previous theorem and of \eqref{eq_ex_multi}.
\begin{remark}
As observed by L.~Silberman and A.~Venkatesh in \cite{SiVe2} and used by V.~Blomer and P.~Maga in \cite{BlMa1} and in \cite{BlMa2}, the precise formulas for the Hall polynomials occurring in \eqref{eq_lin_2} and in \eqref{eq_lin_6} are not really needed for the purpose of the amplification method, since the Hall polynomials are easily well approximated for $p$ and $q$ large by the much easier Schur polynomials. Nevertheless, the precise list of the Hecke operators relevant for the amplification method, namely occurring in \eqref{eq_lin_2} and in \eqref{eq_lin_6}, seems to be crucial in order to obtain the best possible explicit result. For instance, G.~Harcos and N.~Templier used such a list in order to prove the best known subconvexity exponent for the sup-norm of $GL(2)$ automorphic forms in the level aspect in \cite{MR3038127}.
\end{remark}
\subsection{Organization of the paper}%
The general background on $GL(3)$ Maa$\ss$ cusp forms and on the $GL(3)$ Hecke algebra is given in Section \ref{sec_GL3}. The linearizations involved in Theorem \ref{cor_A} are detailed in Section \ref{sec_proof_eq}.  The proof requires decompositions of $\Lambda$-double cosets into $\Lambda$-left and right cosets and computations of degrees as done in Appendix \ref{sec_appendix}.  
\begin{notations}
$\Lambda$ stands for the group $GL_3(\Z)$ of \;$3\times3$ invertible matrices with integer coefficients. If $g$ is a \;$3\times3$ matrix with real coefficients then $\mathstrut^tg$ stands for its transpose. For $g \in GL_3(\Q)$ we let $T_g$ denote the Hecke operator associated to $g$ (see \S \ref{sec_GL3}). If $a$, $b$ and $c$ are three rational numbers then
\begin{itemize}
\item
$\textnormal{diag}(a,b,c)$ denotes the diagonal \;$3\times3$ matrix with $a$, $b$ and $c$ as diagonal coefficients;
\item
$L_{a,b,c}$ (respectively $R_{a,b,c}$) stands for a system of representatives for the decomposition of the $\Lambda$-double coset $\Lambda\textnormal{diag}(a,b,c)\Lambda$ into $\Lambda$-left (respectively right) cosets.
\end{itemize}
\end{notations}
\begin{merci}%
The authors would like to thank the referees for a very careful reading of the manuscript.
\par
They also want to thank J.~Cogdell, \'{E}.~Fouvry, E.~Kowalski and P.~Michel for fruitful discussions related to this work.
\par
This article was worked out at several places: Universit\'e Blaise Pascal (Clermont-Ferrand), the Department of Mathematics of The Ohio State University (Columbus), \'{E}cole Polytechnique F\'ed\'erale de Lausanne, Eidgenossische Technische Hochschule (Zurich), Forschungsinstitut fur Mathematik (Zurich). The authors would like to thank all of these institutions for their hospitality and inspiring working conditions.
\par
The research of R.~Holowinsky was partially supported by the Sloan fellowship BR2011-083 and the NSF grant DMS-1068043.  In-person collaboration with the other authors was made possible through this funding.
\par
The research of G.~Ricotta was partially supported by a Marie Curie Intra European Fellowship within the 7th European Community Framework Programme. The grant agreement number of this project, whose acronym is ANERAUTOHI, is PIEF-GA-2009-25271. He would like to thank ETH and its entire staff for the excellent working conditions. Last but not least, the second author would like to express his gratitude to K.~Belabas for his crucial but isolated support for Analytic Number Theory among the Number Theory research team A2X (Institut de Math\'ematiques de Bordeaux, Universit\'e de Bordeaux).
\end{merci}
\section{Background on the $GL(3)$ Hecke algebra}\label{sec_GL3}%
Convenient references for this section include \cite{MR1349824}, \cite{MR2254662}, \cite{MR0340283} and \cite{MR1291394}.
\par
Let $f$ be a $GL(3)$ Maa$\ss$ cusp form of full level. Such $f$ admits a Fourier expansion
\begin{equation}\label{eq_fourier}
f(g)=\sum_{\gamma\in U_2(\Z)\setminus SL_2(\Z)}\sum_{\substack{m_1\geq 1 \\
m_2\in\Z\setminus\{0\}}}\frac{a_f(m_1,m_2)}{m_1\abs{m_2}}W_{\text{Ja}}\left(\begin{pmatrix}
m_1\abs{m_2} & & \\
& m_1 & \\
& & 1
\end{pmatrix}\begin{pmatrix}
\gamma & \\
& 1
\end{pmatrix}g,\nu_f,\psi_{1,\frac{m_2}{\abs{m_2}}}\right)
\end{equation}
for $g\in GL_3(\R)$ (see \cite[Equation (6.2.1)]{MR2254662}).  Here $U_2(\Z)$ stands for the $\Z$-points of the group of upper-triangular unipotent $2\times 2$ matrices. $\nu_f\in\C^2$ is the type of $f$, whose components are complex numbers characterized by the property that, for every invariant differential operator $D$ in the center of the universal enveloping algebra of $GL_3(\R)$, the cusp form $f$ is an eigenfunction of $D$ with the same eigenvalue as the power function $I_{\nu_f}$, which is defined in \cite[Equation (5.1.1)]{MR2254662}. $\psi_{1,\pm 1}$ is the character of the group of upper-triangular unipotent real $3\times 3$ matrices defined by
\begin{equation*}
\psi_{1,\pm 1}\left(\begin{pmatrix}
1 & u_{1,2} & u_{1,3} \\
& 1 & u_{2,3} \\
& & 1
\end{pmatrix}\right)=e^{2i\pi\left(u_{2,3}\pm u_{1,2}\right)}.
\end{equation*}
$W_{Ja}(\ast,\nu_f,\psi_{1,\pm 1})$ stands for the $GL(3)$ Jacquet Whittaker function of type $\nu_f$ and character $\psi_{1,\pm 1}$ defined in \cite[Equation 6.1.2]{MR2254662}. The complex number $a_f(m_1,m_2)$ is the $(m_1,m_2)$-th \emph{Fourier coefficient} of $f$ for $m_1$ a positive integer and $m_2$ a non-vanishing integer.
\par
For $g\in GL_3(\Q)$, the \emph{Hecke operator} $T_g$ is defined by
\begin{equation*}
T_g(f)(h)=\sum_{\delta\in\Lambda\setminus\Lambda g\Lambda}f(\delta h)
\end{equation*}
for $h\in GL_3(\R)$ (see \cite[Chapter 3, Sections 1.1 and 1.5]{MR1349824}). The \emph{degree} of $g$ or $T_g$ defined by
\begin{equation*}
\textnormal{deg}(g)=\textnormal{deg}(T_g)=\text{card}\left(\Lambda\setminus\Lambda g\Lambda\right)
\end{equation*}
is scaling invariant, in the sense that
\begin{equation}\label{eq_deg_scaling}
\textnormal{deg}(rg)=\textnormal{deg}(g)
\end{equation}
for $r\in\mathbb{Q}^\times$. The adjoint of $T_g$ for the Petersson inner product is $T_{g^{-1}}$. The algebra of Hecke operators $\mathbb{T}$ is the ring of endomorphisms generated by all the $T_g$'s with $g\in GL_3(\Q)$, a commutative algebra of normal endomorphisms (see \cite[Theorem 6.4.6]{MR2254662}), which contains the $m$-th normalized Hecke operator
\begin{equation}\label{eq_Hecke_def}
T_m=\frac{1}{m}\sum_{\substack{g=\textnormal{diag}(y_1,y_2,y_3) \\
y_1\mid y_2\mid y_3\\
y_1y_2y_3=m}}T_g
\end{equation}
for all positive integers $m$. A \emph{Hecke-Maa$\ss$ cusp form} $f$ of full level is a Maa$\ss$ cusp form of full level, which is an eigenfunction of $\mathbb{T}$. In particular, it satisfies
\begin{equation}\label{eq_Hecke}
T_m(f)=a_f(m,1)f \text{ and } T_m^\ast(f)=a_f(1,m)f=\overline{a_f(m,1)}f
\end{equation}
according to \cite[Theorem 6.4.11]{MR2254662}.
\par
The algebra $\mathbb{T}$ is isomorphic to the \emph{absolute Hecke algebra}, the free $\mathbb{Z}$-module generated by the double cosets $\Lambda g\Lambda$ where $g$ ranges over $\Lambda\setminus GL_3(\Q)/\Lambda$ and endowed with the following multiplication law. If $g_1$ and $g_2$ belong to $GL_3(\Q)$ and
\begin{equation*}
\Lambda g_1\Lambda=\bigcup_{i=1}^{\textnormal{deg}(g_1)}\Lambda\alpha_i \text{ and } \Lambda g_2\Lambda=\bigcup_{j=1}^{\textnormal{deg}(g_2)}\Lambda\beta_j 
\end{equation*}
then
\begin{equation}\label{eq_product}
\Lambda g_1\Lambda\ast\Lambda g_2\Lambda=\sum_{\Lambda h\Lambda\subset\Lambda g_1\Lambda g_2\Lambda}m(g_1,g_2;h)\Lambda h\Lambda
\end{equation}
where $h\in GL_3(\Q)$ ranges over a system of representatives of the $\Lambda$-double cosets contained in the set $\Lambda g_1\Lambda g_2\Lambda$ and
\begin{align}\label{eq_multi}
m(g_1,g_2;h) & =\text{card}\left(\left\{(i,j)\in\{1,\dots,\textnormal{deg}(g_1)\}\times\{1,\dots,\textnormal{deg}(g_2)\}, \alpha_i\beta_j\in\Lambda h\right\}\right), \\
& =\frac{1}{\textnormal{deg}(h)}\text{card}\left(\left\{(i,j)\in\{1,\dots,\textnormal{deg}(g_1)\}\times\{1,\dots,\textnormal{deg}(g_2)\}, \alpha_i\beta_j\in\Lambda h\Lambda\right\}\right), \\
& =\frac{\textnormal{deg}(g_2)}{\textnormal{deg}(h)}\text{card}\left(\left\{i\in\{1,\dots,\textnormal{deg}(g_1)\}, \alpha_ig_2\in\Lambda h\Lambda\right\}\right).
\end{align}
Confer \cite[Lemma 1.5 Page 96]{MR1349824}. In particular,
\begin{equation}\label{eq_ex_multi}
\Lambda\textnormal{diag}(r,r,r)\Lambda\ast\Lambda g\Lambda=\Lambda rg\Lambda
\end{equation}
for $g\in GL_3(\Q)$ and $r\in\Q^\times$ (\cite[Lemma 2.4 Page 107]{MR1349824}). In addition, for $p$ and $q$ two distinct prime numbers,
\begin{equation}\label{eq_ex_multi_2}
\Lambda\textnormal{diag}(1,p^{\alpha_1},p^{\alpha_2})\Lambda\ast\Lambda\textnormal{diag}(1,q^{\beta_1},q^{\beta_2})\Lambda=\Lambda\textnormal{diag}(1,p^{\alpha_1}q^{\beta_1},p^{\alpha_2}q^{\beta_2})\Lambda
\end{equation}
where $\alpha_1, \alpha_2, \beta_1,\beta_2$ are non-negative integers by \cite[Proposition 2.5 Page 107]{MR1349824}.
\par
Every double coset $\Lambda g\Lambda$ with $g$ in $GL_3(\Q)$ contains a unique representative of the form
\begin{equation}\label{eq_smith}
r\textnormal{diag}(1,s_1(g),s_2(g))
\end{equation}
where $r\in\Q^\ast$ and $s_1(g), s_2(g)$ are some positive integers satisfying $s_1(g)\mid s_2(g)$ (see \cite[Lemma 2.2]{MR1349824}).
\par
Finally, let $g=[g_{i,j}]_{1\leq i,j\leq 3}$ be a $3\times 3$ matrix with integer coefficients. Its \emph{determinantal divisors} are the non-negative integers given by
\begin{eqnarray*}
d_1(g) & = & \text{gcd}(\{g_{i,j}, 1\leq i,j\leq 3\}), \\
d_2(g) & = & \text{gcd}(\{\text{determinants of $2\times 2$ submatrices of $g$}\}), \\
d_3(g) & = & \abs{\text{det}(g)}.
\end{eqnarray*}
and its \emph{determinantal vector} is $\uple{d}(g)=\left(d_1(g),d_2(g),d_3(g)\right)$. The determinantal divisors turn out to be useful since if $h$ is another $3\times 3$ matrix with integer coefficients then $h$ belongs to $\Lambda g\Lambda$ if and only if $d_k(h)=d_k(g)$ for $1\leq k\leq 3$ (see \cite{MR0340283}).
\section{Proof of the linearizations given in Theorem \ref{cor_A}}\label{sec_proof_eq}%
\subsection{Linearization of $\Lambda\textnormal{diag}(1,1,p)\Lambda\ast\Lambda\textnormal{diag}(1,p,p)\Lambda$}%
This section contains the proof of \eqref{eq_intro_1}.
\par
By \eqref{eq_product}, the product of these double cosets equals
\begin{equation*}
\sum_{\Lambda h\Lambda\subset\Lambda\textnormal{diag}(1,1,p)\Lambda\textnormal{diag}(1,p,p)\Lambda}m\left(\begin{pmatrix}
1 & & \\
& 1 & \\
& & p
\end{pmatrix},\begin{pmatrix}
1 & & \\
& p & \\
& & p
\end{pmatrix};h\right)\Lambda h\Lambda
\end{equation*}
where $h\in GL_3(\Q)$ ranges over a system of representatives of the $\Lambda$-double cosets contained in the set
\begin{equation*}
\Lambda\begin{pmatrix}
1 & & \\
& 1 & \\
& & p
\end{pmatrix}\Lambda\begin{pmatrix}
1 & & \\
& p & \\
& & p
\end{pmatrix}\Lambda.
\end{equation*}
\par
Let us determine the matrices $h$ occuring in this sum. Let $h$ in $GL_3(\Q)$ be such that $\Lambda h\Lambda$ is included in the previous set. By \eqref{eq_smith}, one has uniquely
\begin{equation*}
\Lambda h\Lambda=\Lambda\epsilon\frac{\lambda_1}{\lambda_2}\textnormal{diag}(1,s_1,s_2)\Lambda
\end{equation*}
with $\epsilon=\pm 1$, $\lambda_1, \lambda_2>0$, $(\lambda_1,\lambda_2)=1$, $s_1, s_2>0$, $s_1\mid s_2$. The inclusion is equivalent to
\begin{equation*}
\Lambda\epsilon\lambda_1\textnormal{diag}(1,s_1,s_2)\Lambda=\Lambda\lambda_2\delta_1\delta_2\Lambda
\end{equation*}
for some matrices $\delta_1\in R_{1,1,p}$ and $\delta_2\in L_{1,p,p}$ by \eqref{eq_dec_1_1_p} and \eqref{eq_dec_1_p_p}. So, both matrices have the same determinantal divisors ie
\begin{eqnarray*}
\epsilon\lambda_1 & = & \lambda_2d_1(\delta_1\delta_2), \\
\lambda_1^2s_1 & = & \lambda_2^2d_2(\delta_1\delta_2), \\
\epsilon\lambda_1^3s_1s_2 & = & \lambda_2^3d_3(\delta_1\delta_2)=\lambda_2^3p^3.
\end{eqnarray*}
One can check that the set
\begin{equation*}
\left\{\delta_1\delta_2, (\delta_1,\delta_2)\in R_{1,1,p}\times L_{1,p,p}\right\}
\end{equation*}
is made exactly of the matrices
\begin{eqnarray*}
\begin{pmatrix}
p & & \\
 & p & \\
& & p
\end{pmatrix} & \rightsquigarrow & \uple{d}(\delta_1\delta_2)=(p,p^2,p^3), \\
\begin{pmatrix}
p & d_1+D_1 &  \\
 & p &  \\
& & p
\end{pmatrix} & \rightsquigarrow & \uple{d}(\delta_1\delta_2)=((p,d_1+D_1),p(p,d_1+D_1),p^3), \\
\begin{pmatrix}
p &  & e_1+E_1 \\
 & p & f_1+F_1 \\
& & p
\end{pmatrix} & \rightsquigarrow & \uple{d}(\delta_1\delta_2)=((p,e_1+E_1,f_1+F_1),p(p,e_1+E_1,f_1+F_1),p^3)
\end{eqnarray*}
and
\begin{eqnarray*}
\begin{pmatrix}
p^2 &  & pE_1 \\
 & p & F_1 \\
& & 1
\end{pmatrix} & \rightsquigarrow & \uple{d}(\delta_1\delta_2)=(1,p,p^3), \\
\begin{pmatrix}
p^2 & pD_1 & \\
 & 1 &  \\
& & p
\end{pmatrix} & \rightsquigarrow & \uple{d}(\delta_1\delta_2)=(1,p,p^3), \\
\begin{pmatrix}
1 & pd_1 &  \\
 & p^2 & \\
& & p
\end{pmatrix} & \rightsquigarrow & \uple{d}(\delta_1\delta_2)=(1,p,p^3), \\
\begin{pmatrix}
p & pd_1 & d_1F_1+E_1\\
 & p^2 & pF_1 \\
& & 1
\end{pmatrix} & \rightsquigarrow & \uple{d}(\delta_1\delta_2)=(1,p,p^3), \\
\begin{pmatrix}
1 & & pe_1 \\
 & p & pf_1 \\
& & p^2
\end{pmatrix} & \rightsquigarrow & \uple{d}(\delta_1\delta_2)=(1,p,p^3), \\
\begin{pmatrix}
p & d_2 & pe1 \\
 & 1 & pf_1 \\
& & p^2
\end{pmatrix} & \rightsquigarrow & \uple{d}(\delta_1\delta_2)=(1,p,p^3)
\end{eqnarray*}
with $0\leq d_1, e_1, f_1, D_1, E_1, F_1<p$. As a consequence, only two cases can occur since
\begin{equation*}
\uple{d}(\delta_1\delta_2)\in\{(1,p,p^3),(p,p^2,p^3)\}.
\end{equation*}
\noindent{\underline{First case}}: $\uple{d}(\delta_1\delta_2)=(1,p,p^3)$.
\begin{eqnarray*}
\epsilon\lambda_1 & = & \lambda_2, \\
\lambda_1^2s_1 & = & \lambda_2^2p, \\
\epsilon\lambda_1^3s_1s_2 & = & \lambda_2^3p^3.
\end{eqnarray*}
The first equation gives $\epsilon=\lambda_1=\lambda_2=1$ by the coprimality of $\lambda_1$ and $\lambda_2$. The second equation gives $s_1=p$. The third equation gives $s_2=p^2$. Thus,
\begin{equation*}
\Lambda h\Lambda=\Lambda\textnormal{diag}(1,p,p^2)\Lambda.
\end{equation*}
\noindent{\underline{Second case}}: $\uple{d}(\delta_1\delta_2)=(p,p^2,p^3)$.
\begin{eqnarray*}
\epsilon\lambda_1 & = & \lambda_2p, \\
\lambda_1^2s_1 & = & \lambda_2^2p^2, \\
\epsilon\lambda_1^3s_1s_2 & = & \lambda_2^3p^3.
\end{eqnarray*}
The first equation gives $\epsilon=\lambda_2=1$ and $\lambda_1=p$ by the coprimality of $\lambda_1$ and $\lambda_2$. The second equation gives $s_1=1$. The third equation gives $s_2=1$. Thus,
\begin{equation*}
\Lambda h\Lambda=\Lambda\textnormal{diag}(p,p,p)\Lambda.
\end{equation*}
As a consequence,
\begin{equation*}
\Lambda\textnormal{diag}(1,1,p)\Lambda\ast\Lambda\textnormal{diag}(1,p,p)\Lambda=m_1\Lambda\begin{pmatrix}
1 & & \\
& p & \\
& & p^2
\end{pmatrix}\Lambda+m_2\Lambda\begin{pmatrix}
p & & \\
& p & \\
& & p
\end{pmatrix}\Lambda
\end{equation*}
where
\begin{eqnarray*}
m_1 & \coloneqq & m\left(\begin{pmatrix}
1 & & \\
& 1 & \\
& & p
\end{pmatrix},\begin{pmatrix}
1 & & \\
& p & \\
& & p
\end{pmatrix};\begin{pmatrix}
1 & & \\
& p & \\
& & p^2
\end{pmatrix}\right), \\
m_2 & \coloneqq & m\left(\begin{pmatrix}
1 & & \\
& 1 & \\
& & p
\end{pmatrix},\begin{pmatrix}
1 & & \\
& p & \\
& & p
\end{pmatrix};\begin{pmatrix}
p & & \\
& p & \\
& & p
\end{pmatrix}\right).
\end{eqnarray*}
\par
Let us compute the value of $m_2$ first. By \eqref{eq_multi}, \eqref{eq_deg_1_p_p} and \eqref{eq_deg_scaling},
\begin{equation*}
m_2=(p^2+p+1)\left\vert\left\{\delta_1\in R_{1,1,p}, \delta_1\begin{pmatrix}
1 & & \\
& p & \\
& & p
\end{pmatrix}\in\Lambda\begin{pmatrix}
p & & \\
& p & \\
& & p
\end{pmatrix}\Lambda\right\}\right\vert.
\end{equation*}
Let us compute the remaining cardinality. One can check that the set
\begin{equation*}
\left\{\delta_1\begin{pmatrix}
1 & & \\
& p & \\
& & p
\end{pmatrix}, \delta_1\in R_{1,1,p}\right\}
\end{equation*}
is exactly made of the matrices
\begin{eqnarray*}
\begin{pmatrix}
p & & \\
& p & \\
& & p
\end{pmatrix} & \rightsquigarrow & (d_1,d_2,d_3)=(p,p^2,p^3), \\
\begin{pmatrix}
1 & pd_1 & \\
& p^2 & \\
& & p
\end{pmatrix} & \rightsquigarrow & (d_1,d_2,d_3)=(1,p,p^3), \\
\begin{pmatrix}
1 & 0 & pe_1 \\
& p & pf_1 \\
& & p^2
\end{pmatrix} & \rightsquigarrow & (d_1,d_2,d_3)=(1,p,p^3)
\end{eqnarray*}
with $0\leq d_1, e_1, f_1<p$. The fact that the determinantal vector of $\textnormal{diag}(p,p,p)$ is $\textnormal{diag}(p,p^2,p^3)$ implies that
\begin{equation*}
\left\{\delta_1\in R_{1,1,p}, \delta_1\begin{pmatrix}
1 & & \\
& p & \\
& & p
\end{pmatrix}\in\Lambda\begin{pmatrix}
p & & \\
& p & \\
& & p
\end{pmatrix}\Lambda\right\}=\left\{\begin{pmatrix}
p & & \\
& p & \\
& & p
\end{pmatrix}\right\}
\end{equation*}
and is of cardinality $1$ such that $m_2=p^2+p+1$.
\par
Now, let us compute the value of $m_1$. By \eqref{eq_multi}, \eqref{eq_deg_1_p_p} and \eqref{eq_deg_1_p_p^2},
\begin{equation*}
m_1=\frac{1}{p(p+1)}\left\vert\left\{\delta_1\in R_{1,1,p}, \delta_1\begin{pmatrix}
1 & & \\
& p & \\
& & p
\end{pmatrix}\in\Lambda\begin{pmatrix}
1 & & \\
& p & \\
& & p^2
\end{pmatrix}\Lambda\right\}\right\vert.
\end{equation*}
Let us compute the remaining cardinality. Both the analysis done for $m_2$ and the fact that the determinantal vector of $\textnormal{diag}(1,p,p^2)$ is $(1,p,p^2)$ imply that 
\begin{multline*}
\left\{\delta_1\in R_{1,1,p}, \delta_1\begin{pmatrix}
1 & & \\
& p & \\
& & p
\end{pmatrix}\in\Lambda\begin{pmatrix}
1 & & \\
& p & \\
& & p^2
\end{pmatrix}\Lambda\right\}=\bigcup_{0\leq d_1<p}\left\{\begin{pmatrix}
1 &  d_1 & \\
& p & \\
& & 1
\end{pmatrix}\right\} \\
\bigcup_{0\leq e_1, f_1<p}\left\{\begin{pmatrix}
1 &   & e_1 \\
& 1 & f_1 \\
& & p
\end{pmatrix}\right\}
\end{multline*}
which is of cardinality $p(p+1)$ such that $m_1=1$.
\subsection{Linearization of $\Lambda\textnormal{diag}(1,p,p^2)\Lambda\ast\Lambda\textnormal{diag}(1,p,p^2)\Lambda$}%
This section contains the proof of \eqref{eq_intro_2}.
\par
By \eqref{eq_product}, the product of these double cosets equals
\begin{equation*}
\sum_{\Lambda h\Lambda\subset\Lambda\textnormal{diag}(1,p,p^2)\Lambda \textnormal{diag}(1,p,p^2)\Lambda}m\left(\begin{pmatrix}
1 & & \\
& p & \\
& & p^2
\end{pmatrix},\begin{pmatrix}
1 & & \\
& p & \\
& & p^2
\end{pmatrix};h\right)\Lambda h\Lambda
\end{equation*}
where $h\in GL_3(\Q)$ ranges over a system of representatives of the $\Lambda$-double cosets contained in the set
\begin{equation*}
\Lambda\begin{pmatrix}
1 & & \\
& p & \\
& & p^2
\end{pmatrix}\Lambda\begin{pmatrix}
1 & & \\
& p & \\
& & p^2
\end{pmatrix}\Lambda.
\end{equation*}
\par
Let us determine the relevant matrices $h$ occuring in this sum. Let $h$ in $GL_3(\Q)$ be such that $\Lambda h\Lambda$ is included in the previous set. By \eqref{eq_smith}, one has uniquely
\begin{equation*}
\Lambda h\Lambda=\Lambda\epsilon\frac{\lambda_1}{\lambda_2}\textnormal{diag}(1,s_1,s_2)\Lambda
\end{equation*}
with $\epsilon=\pm 1$, $\lambda_1, \lambda_2>0$, $(\lambda_1,\lambda_2)=1$, $s_1, s_2>0$, $s_1\mid s_2$. The inclusion is equivalent to
\begin{equation*}
\Lambda\epsilon\lambda_1\textnormal{diag}(1,s_1,s_2)\Lambda=\Lambda\lambda_2\delta_1\delta_2\Lambda
\end{equation*}
for some matrices $\delta_1\in R_{1,p,p^2}$ and $\delta_2\in L_{1,p,p^2}$ by \eqref{eq_dec_1_p_p^2}. So, both matrices have the same determinantal divisors ie
\begin{eqnarray*}
\epsilon\lambda_1 & = & \lambda_2d_1(\delta_1\delta_2), \\
\lambda_1^2s_1 & = & \lambda_2^2d_2(\delta_1\delta_2), \\
\epsilon\lambda_1^3s_1s_2 & = & \lambda_2^3d_3(\delta_1\delta_2)=\lambda_2^3p^6.
\end{eqnarray*}
By \eqref{eq_dec_1_p_p^2}, a straightforward but tedious computation ensures that the set
\begin{equation*}
\left\{\uple{d}(\delta_1\delta_2), (\delta_1,\delta_2)\in R_{1,p,p^2}\times L_{1,p,p^2}\right\}
\end{equation*}
is a subset of
\begin{equation*}
\{(1,p^2,p^6),(1,p^3,p^6),(p,p^2,p^6),(p,p^3,p^6),(p^2,p^4,p^6)\}.
\end{equation*}
\noindent{\underline{Case 1}}: $(d_1,d_2,d_3)=(1,p^2,p^6)$.
\begin{eqnarray*}
\epsilon\lambda_1 & = & \lambda_2, \\
\lambda_1^2s_1 & = & \lambda_2^2p^2, \\
\epsilon\lambda_1^3s_1s_2 & = & \lambda_2^3p^6.
\end{eqnarray*}
The first equation gives $\epsilon=\lambda_1=\lambda_2=1$ by the coprimality of $\lambda_1$ and $\lambda_2$. The second equation gives $s_1=p^2$. The third equation gives $s_2=p^4$. Thus,
\begin{equation*}
\Lambda h\Lambda=\Lambda\textnormal{diag}(1,p^2,p^4)\Lambda.
\end{equation*}
\noindent{\underline{Case 2}}: $(d_1,d_2,d_3)=(1,p^3,p^6)$.
\begin{eqnarray*}
\epsilon\lambda_1 & = & \lambda_2, \\
\lambda_1^2s_1 & = & \lambda_2^2p^3, \\
\epsilon\lambda_1^3s_1s_2 & = & \lambda_2^3p^6.
\end{eqnarray*}
The first equation gives $\epsilon=\lambda_1=\lambda_2=1$ by the coprimality of $\lambda_1$ and $\lambda_2$. The second equation gives $s_1=p^3$. The third equation gives $s_2=p^3$. Thus,
\begin{equation*}
\Lambda h\Lambda=\Lambda\textnormal{diag}(1,p^3,p^3)\Lambda.
\end{equation*}
\noindent{\underline{Case 3}}: $(d_1,d_2,d_3)=(p,p^2,p^6)$.
\begin{eqnarray*}
\epsilon\lambda_1 & = & \lambda_2p, \\
\lambda_1^2s_1 & = & \lambda_2^2p^2, \\
\epsilon\lambda_1^3s_1s_2 & = & \lambda_2^3p^6.
\end{eqnarray*}
The first equation gives $\epsilon=\lambda_2=1$ and $\lambda_1=p$ by the coprimality of $\lambda_1$ and $\lambda_2$. The second equation gives $s_1=1$. The third equation gives $s_2=p^3$. Thus,
\begin{equation*}
\Lambda h\Lambda=\Lambda\textnormal{diag}(p,p,p^4)\Lambda.
\end{equation*}
\noindent{\underline{Case 4}}: $(d_1,d_2,d_3)=(p,p^3,p^6)$.
\begin{eqnarray*}
\epsilon\lambda_1 & = & \lambda_2p, \\
\lambda_1^2s_1 & = & \lambda_2^2p^3, \\
\epsilon\lambda_1^3s_1s_2 & = & \lambda_2^3p^6.
\end{eqnarray*}
The first equation gives $\epsilon=\lambda_2=1$ and $\lambda_1=p$ by the coprimality of $\lambda_1$ and $\lambda_2$. The second equation gives $s_1=p$. The third equation gives $s_2=p^2$. Thus,
\begin{equation*}
\Lambda h\Lambda=\Lambda\textnormal{diag}(p,p^2,p^3)\Lambda.
\end{equation*}
\noindent{\underline{Case 5}}: $(d_1,d_2,d_3)=(p^2,p^4,p^6)$.
\begin{eqnarray*}
\epsilon\lambda_1 & = & \lambda_2p^2, \\
\lambda_1^2s_1 & = & \lambda_2^2p^4, \\
\epsilon\lambda_1^3s_1s_2 & = & \lambda_2^3p^6.
\end{eqnarray*}
The first equation gives $\epsilon=\lambda_2=1$ and $\lambda_1=p^2$ by the coprimality of $\lambda_1$ and $\lambda_2$. The second equation gives $s_1=1$. The third equation gives $s_2=1$. Thus,
\begin{equation*}
\Lambda h\Lambda=\Lambda\textnormal{diag}(p^2,p^2,p^2)\Lambda.
\end{equation*}
As a consequence,
\begin{multline*}
\Lambda\textnormal{diag}(1,p,p^2)\Lambda\ast\Lambda\textnormal{diag}(1,p,p^2)\Lambda=m_1\Lambda\begin{pmatrix}
1 & & \\
& p^2 & \\
& & p^4
\end{pmatrix}\Lambda+m_2\Lambda\begin{pmatrix}
1 & & \\
& p^3 & \\
& & p^3
\end{pmatrix}\Lambda \\
+m_3\Lambda\begin{pmatrix}
p & & \\
& p & \\
& & p^4
\end{pmatrix}\Lambda+m_4\Lambda\begin{pmatrix}
p & & \\
& p^2 & \\
& & p^3
\end{pmatrix}\Lambda+m_5\Lambda\begin{pmatrix}
p^2 & & \\
& p^2 & \\
& & p^2
\end{pmatrix}\Lambda
\end{multline*}
where
\begin{eqnarray*}
m_1 & \coloneqq & m\left(\begin{pmatrix}
1 & & \\
& p & \\
& & p^2
\end{pmatrix},\begin{pmatrix}
1 & & \\
& p & \\
& & p^2
\end{pmatrix};\begin{pmatrix}
1 & & \\
& p^2 & \\
& & p^4
\end{pmatrix}\right), \\
m_2 & \coloneqq & m\left(\begin{pmatrix}
1 & & \\
& p & \\
& & p^2
\end{pmatrix},\begin{pmatrix}
1 & & \\
& p & \\
& & p^2
\end{pmatrix};\begin{pmatrix}
1 & & \\
& p^3 & \\
& & p^3
\end{pmatrix}\right), \\
m_3 & \coloneqq & m\left(\begin{pmatrix}
1 & & \\
& p & \\
& & p^2
\end{pmatrix},\begin{pmatrix}
1 & & \\
& p & \\
& & p^2
\end{pmatrix};\begin{pmatrix}
p & & \\
& p & \\
& & p^4
\end{pmatrix}\right), \\
m_4 & \coloneqq & m\left(\begin{pmatrix}
1 & & \\
& p & \\
& & p^2
\end{pmatrix},\begin{pmatrix}
1 & & \\
& p & \\
& & p^2
\end{pmatrix};\begin{pmatrix}
p & & \\
& p^2 & \\
& & p^3
\end{pmatrix}\right), \\
m_5 & \coloneqq & m\left(\begin{pmatrix}
1 & & \\
& p & \\
& & p^2
\end{pmatrix},\begin{pmatrix}
1 & & \\
& p & \\
& & p^2
\end{pmatrix};\begin{pmatrix}
p^2 & & \\
& p^2 & \\
& & p^2
\end{pmatrix}\right).
\end{eqnarray*}
\par
Let us compute the value of $m_1$. By \eqref{eq_multi}, \eqref{eq_deg_1_p_p^2} and \eqref{eq_deg_1_p^2_p^4},
\begin{equation*}
m_1=\frac{1}{p^4}\left\vert\left\{\delta_1\in R_{1,p,p^2}, \delta_1\begin{pmatrix}
1 & & \\
& p & \\
& & p^2
\end{pmatrix}\in\Lambda\begin{pmatrix}
1 & & \\
& p^2 & \\
& & p^4
\end{pmatrix}\Lambda\right\}\right\vert.
\end{equation*}
Let us compute the remaining cardinality. One can check that the set
\begin{equation*}
\left\{\delta_1\begin{pmatrix}
1 & & \\
& p & \\
& & p^2
\end{pmatrix}, \delta_1\in R_{1,p,p^2}\right\}
\end{equation*}
is exactly made of the matrices
\begin{eqnarray*}
\begin{pmatrix}
p^2 &  &  \\
& p & p^2f_1 \\
& & p^3
\end{pmatrix} & \rightsquigarrow & (d_1,d_2,d_3)=(p,p^3,p^6), \\
\begin{pmatrix}
p & pd_2 &  \\
& p^3 &  \\
& & p^2
\end{pmatrix} (p\mid d_2) & \rightsquigarrow & (d_1,d_2,d_3)=(p,p^3,p^6), \\
\begin{pmatrix}
p & pd_1 & p^2e_1 \\
& p^2 & p^2f_1 \\
& & p^3
\end{pmatrix} (d_1f_1=0, (d_1,e_1,f_1)\neq(0,0,0)) & \rightsquigarrow & (d_1,d_2,d_3)=(p,p^3,p^6)
\end{eqnarray*}
and
\begin{eqnarray*}
\begin{pmatrix}
1 & pd_1 & p^2e_2 \\
& p^2 & p^2f_2 \\
& & p^4
\end{pmatrix} (p\mid f_2) & \rightsquigarrow & (d_1,d_2,d_3)=(1,p^2,p^6)
\end{eqnarray*}
and
\begin{eqnarray*}
\begin{pmatrix}
1 & pd_2 & p^2e_1 \\
& p^3 &  \\
& & p^3
\end{pmatrix} & \rightsquigarrow & (d_1,d_2,d_3)=(1,p^3,p^6)
\end{eqnarray*}
and
\begin{eqnarray*}
\begin{pmatrix}
p &  & p^2e_2 \\
& p & p^2f_2 \\
& & p^4
\end{pmatrix} (p\mid e_2) & \rightsquigarrow & (d_1,d_2,d_3)=(p,p^2,p^6)
\end{eqnarray*}
and
\begin{eqnarray*}
\begin{pmatrix}
p^2 &  &  \\
& p^2 &  \\
& & p^2
\end{pmatrix} & \rightsquigarrow & (d_1,d_2,d_3)=(p^2,p^4,p^6)
\end{eqnarray*}
where $0\leq d_1, e_1, f_1<p$ and $0\leq d_2, e_2, f_2<p^2$. The fact that the determinantal vector of $\textnormal{diag}(1,p^2,p^4)$ is $(1,p^2,p^6)$ implies that $m_1=1$.
\par
Let us compute the value of $m_2$. By \eqref{eq_multi}, \eqref{eq_deg_1_p_p^2} and \eqref{eq_deg_1_p^3_p^3},
\begin{equation*}
m_2=\frac{p+1}{p^3}\left\vert\left\{\delta_1\in R_{1,p,p^2}, \delta_1\begin{pmatrix}
1 & & \\
& p & \\
& & p^2
\end{pmatrix}\in\Lambda\begin{pmatrix}
1 & & \\
& p^3 & \\
& & p^3
\end{pmatrix}\Lambda\right\}\right\vert.
\end{equation*}
Both the analysis done for $m_1$ and the fact that the determinantal vector of $\textnormal{diag}(1,p^3,p^3)$ is $(1,p^3,p^6)$ imply that $m_2=p+1$.
\par
Let us compute the value of $m_3$. By \eqref{eq_multi}, \eqref{eq_deg_1_p_p^2}, \eqref{eq_deg_scaling} and \eqref{eq_deg_p_p_p^4},
\begin{equation*}
m_3=\frac{p+1}{p^3}\left\vert\left\{\delta_1\in R_{1,p,p^2}, \delta_1\begin{pmatrix}
1 & & \\
& p & \\
& & p^2
\end{pmatrix}\in\Lambda\begin{pmatrix}
p & & \\
& p & \\
& & p^4
\end{pmatrix}\Lambda\right\}\right\vert.
\end{equation*}
Both the analysis done for $m_1$ and the fact that the determinantal vector of $\textnormal{diag}(p,p,p^4)$ is $(p,p^2,p^6)$ imply that $m_3=p+1$.
\par
Let us compute the value of $m_4$. By \eqref{eq_multi}, \eqref{eq_deg_1_p_p^2}, \eqref{eq_deg_scaling} and \eqref{eq_deg_1_p_p^2},
\begin{equation*}
m_4=\frac{p+1}{p^3}\left\vert\left\{\delta_1\in R_{1,p,p^2}, \delta_1\begin{pmatrix}
1 & & \\
& p & \\
& & p^2
\end{pmatrix}\in\Lambda\begin{pmatrix}
p & & \\
& p^2 & \\
& & p^3
\end{pmatrix}\Lambda\right\}\right\vert.
\end{equation*}
Both the analysis done for $m_1$ and the fact that the determinantal vector of $\textnormal{diag}(p,p^2,p^3)$ is $(p,p^3,p^6)$ imply that $m_4=(p+1)(2p-1)$.
\par
Let us compute the value of $m_5$. By \eqref{eq_multi}, \eqref{eq_deg_1_p_p^2} and \eqref{eq_deg_scaling},
\begin{equation*}
m_5=p(p+1)(p^2+p+1)\left\vert\left\{\delta_1\in R_{1,p,p^2}, \delta_1\begin{pmatrix}
1 & & \\
& p & \\
& & p^2
\end{pmatrix}\in\Lambda\begin{pmatrix}
p^2 & & \\
& p^2 & \\
& & p^2
\end{pmatrix}\Lambda\right\}\right\vert.
\end{equation*}
Both the analysis done for $m_1$ and the fact that the determinantal vector of $\textnormal{diag}(p^2,p^2,p^2)$ is $(p^2,p^4,p^6)$ imply that $m_5=p(p+1)(p^2+p+1)$.
\appendix
\section{Decomposition of $\Lambda$-double cosets into $\Lambda$-cosets}\label{sec_appendix}%
By \cite[Lemma 1.2 Page 94 and Lemma 2.1 Page 105]{MR1349824}, we know that every $\Lambda$-double coset $\Lambda g\Lambda$ with $g$ in $GL_3(\Q)$ with integer coefficients is both a finite union of $\Lambda$-left cosets and $\Lambda$-right cosets. In addition, every $\Lambda$-right coset $\Lambda g$ contains a unique upper-triangular column reduced representative, namely
\begin{equation*}
\Lambda g=\Lambda\begin{pmatrix}
a & d & e \\
& b & f \\
& & c
\end{pmatrix}
\end{equation*}
where $0\leq d<b$ and $0\leq e, f<c$ by \cite[Lemma 2.7 Page 109]{MR1349824}.
\par
As a consequence, every $\Lambda$-left coset $g\Lambda$ contains a unique upper-triangular row reduced representative, namely
\begin{equation*}
g \Lambda=\begin{pmatrix}
a & d & e \\
 & b & f \\
 &  & c
\end{pmatrix}\Lambda
\end{equation*}
where $0\leq d,e <a$ and $0\leq f<b$. More explicitely, if $UW\mathstrut^{t}gW=H$ is the upper-triangular column reduced representative of the $\Lambda$-right coset $\Lambda W\mathstrut^{t}gW$ with $W$ the anti-diagonal matrix with $1$'s on the anti-diagonal then $gW\mathstrut^{t}UW=W\mathstrut^{t}HW$ is the upper-triangular row reduced representative of the $\Lambda$-left coset $g\Lambda$.
\par
The previous fact also entails that
\begin{equation}\label{eq_left_right}
\Lambda g\Lambda=\bigcup_{\delta\in R_g}\Lambda\delta\Rightarrow\Lambda g\Lambda=\bigcup_{\delta\in W\mathstrut^{t}R_gW}\delta\Lambda
\end{equation}
since
\begin{equation*}
\Lambda g\Lambda=W\Lambda g\Lambda=W\mathstrut^{t}\left(\Lambda g\Lambda\right)=W\bigcup_{\delta\in R_g}\mathstrut^{t}\delta\Lambda=\bigcup_{\delta\in R_g}W\mathstrut^{t}\delta W\Lambda.
\end{equation*}
\par
Let us finish with a useful elementary practical remark for the computations done in the following sections of the appendix. If $H$ is an upper-triangular column reduced matrix in a $\Lambda$-double coset $\Lambda\textnormal{diag}(p^{\alpha_1},p^{\alpha_2},p^{\alpha_3})\Lambda$ where $p$ is a prime number and $\alpha_1$, $\alpha_2$ and $\alpha_3$ are non-negative integers then
\begin{equation}\label{eq_practice}
H=\begin{pmatrix}
p^{\delta_1} & \ast & \ast \\
& p^{\delta_2} & \ast \\
& & p^{\delta_3}
\end{pmatrix}, \sum_{j=1}^3(\alpha_j-\delta_j)=0, \forall j\in\{1,2,3\}, 0\leq\delta_j\leq\max_{1\leq k\leq 3}{\alpha_k}.
\end{equation}
The fact that the diagonal cofficients of $H$ are powers of $p$ comes from the determinant equation. The condition on the exponents of these diagonal coefficients follows from the fact that $p^{\max{\left\{\alpha_k, 1\leq k\leq 3\right\}}}H^{-1}$ has integer coefficients.
\subsection{Decomposition and degree of $\Lambda \textnormal{diag}(1,1,p)\Lambda$}%
\begin{proposition}\label{propo_dec_1_1_p}
One has
\begin{equation}\label{eq_dec_1_1_p}
\Lambda\textnormal{diag}(1,1,p)\Lambda=\bigcup_{\delta\in R_{1,1,p}}\Lambda\delta=\bigcup_{\delta\in L_{1,1,p}}\delta\Lambda
\end{equation}
where
\begin{equation*}
R_{1,1,p}=\left\{\textnormal{diag}(p,1,1)\right\}\bigcup_{0\leq d_1<p}\left\{\begin{pmatrix}
1 & d_1 & \\
& p & \\
& & 1
\end{pmatrix}\right\}\bigcup_{0\leq e_1,f_1<p}\left\{\begin{pmatrix}
1 & 0 & e_1 \\
& 1 & f_1 \\
& & p
\end{pmatrix}\right\}
\end{equation*}
and
\begin{equation*}
L_{1,1,p}=\left\{\textnormal{diag}(1,1,p)\right\}\bigcup_{0\leq f_1<p}\left\{\begin{pmatrix}
1 &  & \\
& p & f_1 \\
& & 1
\end{pmatrix}\right\}\bigcup_{0\leq d_1, e_1<p}\left\{\begin{pmatrix}
p & d_1 & e_1 \\
& 1 &  \\
& & 1
\end{pmatrix}\right\}.
\end{equation*}
In particular,
\begin{equation}\label{eq_deg_1_1_p}
\textnormal{deg}\left(\textnormal{diag}(1,1,p)\right)=p^2+p+1.
\end{equation}
\end{proposition}
\begin{proof}[\proofname{} of Proposition \ref{propo_dec_1_1_p}]%
The decomposition into $\Lambda$-right cosets implies the decomposition into $\Lambda$-left cosets by \eqref{eq_left_right}. The possible upper-triangular column reduced matrices $\delta$ that can occur in the decomposition into $\Lambda$-right cosets are
\begin{eqnarray*}
\begin{pmatrix}
1 & 0 & e_1 \\
& 1 & f_1 \\
& & p
\end{pmatrix} & \rightsquigarrow & \uple{d}(\delta)=(1,1,p), \\
\begin{pmatrix}
1 & d_1 & 0 \\
& p & 0 \\
& & 1
\end{pmatrix} & \rightsquigarrow & \uple{d}(\delta)=(1,1,p), \\
\begin{pmatrix}
p & 0 & 0 \\
& 1 & 0 \\
& & 1
\end{pmatrix} & \rightsquigarrow & \uple{d}(\delta)=(1,1,p)
\end{eqnarray*}
where $0\leq d_1, e_1, f_1<p$. The fact that the determinantal vector of $\textnormal{diag}(1,1,p)$ is $(1,1,p)$ implies the decomposition into $\Lambda$-left cosets given in \eqref{eq_dec_1_1_p} and the computation of the degree too.
\end{proof}
\subsection{Decomposition and degree of $\Lambda\textnormal{diag}(1,p,p)\Lambda$}%
\begin{proposition}\label{propo_dec_1_p_p}
One has
\begin{equation}\label{eq_dec_1_p_p}
\Lambda\textnormal{diag}(1,p,p)\Lambda=\cup_{\delta\in L_{1,p,p}}\delta\Lambda
\end{equation}
where
\begin{equation*}
L_{1,p,p}=\left\{\textnormal{diag}(1,p,p)\right\}\bigcup_{0\leq e_1, f_1<p}\left\{\begin{pmatrix}
p &  & e_1 \\
& p & f_1 \\
& & 1
\end{pmatrix}\right\}\bigcup_{0\leq d_1<p}\left\{\begin{pmatrix}
p & d_1 &  \\
& 1 &  \\
& & p
\end{pmatrix}\right\}
\end{equation*}
In particular,
\begin{equation}\label{eq_deg_1_p_p}
\textnormal{deg}\left(\textnormal{diag}(1,p,p)\right)=p^2+p+1.
\end{equation}
\end{proposition}
\begin{proof}[\proofname{} of Proposition \ref{propo_dec_1_p_p}]%
By \eqref{eq_practice}, the possible upper-triangular row reduced matrices $\delta$ that can occur in the decomposition into $\Lambda$-left cosets are
\begin{eqnarray*}
\begin{pmatrix}
p & d_1 & e_1 \\
& p & f_1 \\
& & 1
\end{pmatrix} & \rightsquigarrow & \uple{d}(\delta)=(1,(p,d_1,d_1f_1),p^2), \\
\begin{pmatrix}
p & d_1 & e_1 \\
& 1 &  \\
& & p
\end{pmatrix} & \rightsquigarrow & \uple{d}(\delta)=(1,(p,e_1),p^2), \\
\begin{pmatrix}
1 &  &  \\
& p & f_1 \\
& & p
\end{pmatrix} & \rightsquigarrow & \uple{d}(\delta)=(1,(p,f_1),p^2)
\end{eqnarray*}
where $0\leq d_1, e_1, f_1<p$. The fact that the determinantal vector of $\textnormal{diag}(1,p,p)$ is $(1,p,p^2)$ implies the decomposition into $\Lambda$-left cosets given in \eqref{eq_dec_1_p_p} and the computation of the degree too.
\end{proof}
\subsection{Decomposition and degree of $\Lambda\textnormal{diag}(1,p,p^2)\Lambda$}%
\begin{proposition}\label{propo_dec_1_p_p^2}
One has
\begin{equation}\label{eq_dec_1_p_p^2}
\Lambda\textnormal{diag}(1,p,p^2)\Lambda=\cup_{\delta\in R_{1,p,p^2}}\Lambda\delta=\cup_{\delta\in L_{1,p,p^2}}\delta\Lambda
\end{equation}
where
\begin{multline*}
R_{1,p,p^2}=\bigcup_{\substack{0\leq d_1<p \\
0\leq e_2, f_2 <p^2 \\
p\mid f_2}}\left\{\begin{pmatrix}
1 & d_1 & e_2 \\
& p & f_2 \\
& & p^2
\end{pmatrix}\right\}\bigcup_{\substack{0\leq e_1<p \\
0\leq d_2 <p^2}}\left\{\begin{pmatrix}
1 & d_2 & e_1 \\
& p^2 &  \\
& & p
\end{pmatrix}\right\} \\
\bigcup_{\substack{0\leq e_2, f_2 <p^2 \\
p\mid e_2}}\left\{\begin{pmatrix}
p &  & e_2 \\
& 1 & f_2 \\
& & p^2
\end{pmatrix}\right\}\bigcup_{0\leq f_1<p}\left\{\begin{pmatrix}
p^2 &  &  \\
& 1 & f_1 \\
& & p
\end{pmatrix}\right\}\bigcup_{\substack{0\leq d_2<p^2 \\
p\mid d_2}}\left\{\begin{pmatrix}
p & d_2 &  \\
& p^2 &  \\
& & 1
\end{pmatrix}\right\} \\
\bigcup_{\substack{0\leq d_1, e_1, f_1<p \\
d_1f_1=0 \\
(d_1,e_1,f_1)\neq (0,0,0)}}\left\{\begin{pmatrix}
p & d_1 & e_1 \\
& p & f_1 \\
& & p
\end{pmatrix}\right\}\bigcup\left\{\begin{pmatrix}
p^2 &  &  \\
& p &  \\
& & 1
\end{pmatrix}\right\}
\end{multline*}
and
\begin{multline*}
L_{1,p,p^2}=\bigcup_{\substack{0\leq f_1<p \\
0\leq d_2, e_2 <p^2 \\
p\mid d_2}}\left\{\begin{pmatrix}
p^2 & d_2 & e_2 \\
& p & f_1 \\
& & 1
\end{pmatrix}\right\}\bigcup_{\substack{0\leq e_1<p \\
0\leq f_2 <p^2}}\left\{\begin{pmatrix}
p &  & e_1 \\
& p^2 & f_2 \\
& & 1
\end{pmatrix}\right\} \\
\bigcup_{\substack{0\leq d_2, e_2 <p^2 \\
p\mid e_2}}\left\{\begin{pmatrix}
p^2 & d_2 & e_2 \\
& 1 &  \\
& & p
\end{pmatrix}\right\}\bigcup_{0\leq d_1<p}\left\{\begin{pmatrix}
p & d_1 &  \\
& 1 &  \\
& & p^2
\end{pmatrix}\right\}\bigcup_{\substack{0\leq f_2<p^2 \\
p\mid f_2}}\left\{\begin{pmatrix}
1 &  &  \\
& p^2 & f_2 \\
& & p
\end{pmatrix}\right\} \\
\bigcup_{\substack{0\leq d_1, e_1, f_1<p \\
d_1f_1=0 \\
(d_1,e_1,f_1)\neq (0,0,0)}}\left\{\begin{pmatrix}
p & d_1 & e_1 \\
& p & f_1 \\
& & p
\end{pmatrix}\right\}\bigcup\left\{\begin{pmatrix}
p^2 &  &  \\
& p &  \\
& & 1
\end{pmatrix}\right\}.
\end{multline*}
In particular,
\begin{equation}\label{eq_deg_1_p_p^2}
\textnormal{deg}\left(\textnormal{diag}(1,p,p^2)\right)=p(p+1)(1+p+p^2).
\end{equation}
\end{proposition}
\begin{proof}[\proofname{} of Proposition \ref{propo_dec_1_p_p^2}]%
The decomposition into $\Lambda$-right cosets implies the decomposition into $\Lambda$-left cosets by \eqref{eq_left_right}. By \eqref{eq_practice}, the possible upper-triangular column reduced matrices $\delta$ that can occur in the decomposition into $\Lambda$-right cosets are
\begin{eqnarray*}
\text{Type $1$:} \begin{pmatrix}
p & d_1 & e_1 \\
& p & f_1 \\
& & p
\end{pmatrix} & \rightsquigarrow & \uple{d}(\delta)=((p,d_1,e_1,f_1),(p^2,pd_1,pf_1,d_1f_1-pe_1),p^3)
\end{eqnarray*}
and
\begin{eqnarray*}
\text{Type $2$:} \begin{pmatrix}
1 & d_1 & e_2 \\
& p & f_2 \\
& & p^2
\end{pmatrix} & \rightsquigarrow & \uple{d}(\delta)=(1,(p,f_2),p^3), \\
\text{Type $3$:} \begin{pmatrix}
1 & d_2 & e_1 \\
& p^2 & f_1 \\
& & p
\end{pmatrix} & \rightsquigarrow & \uple{d}(\delta)=(1,(p,f_1),p^3), \\
\text{Type $4$:} \begin{pmatrix}
p &  & e_2 \\
& 1 & f_2 \\
& & p^2
\end{pmatrix} & \rightsquigarrow & \uple{d}(\delta)=(1,(p,e_2),p^3), \\
\text{Type $5$:} \begin{pmatrix}
p^2 &  & e_1 \\
& 1 & f_1 \\
& & p
\end{pmatrix} & \rightsquigarrow & \uple{d}(\delta)=(1,(p,e_1),p^3), \\
\text{Type $6$:} \begin{pmatrix}
p & d_2 &  \\
& p^2 &  \\
& & 1
\end{pmatrix} & \rightsquigarrow & \uple{d}(\delta)=(1,(p,d_2),p^3), \\
\text{Type $7$:} \begin{pmatrix}
p^2 & d_1 &  \\
& p &  \\
& & 1
\end{pmatrix} & \rightsquigarrow & \uple{d}(\delta)=(1,(p,d_1),p^3)
\end{eqnarray*}
where $0\leq d_1,e_1,f_1<p$ and $0\leq d_2,e_2,f_2<p^2$. Let us count the matrices among the previous ones, whose determinantal vector is the same as the one of $\textnormal{diag}(1,p,p^2)$, namely $(1,p,p^3)$.
\par
Let us consider the matrices of type $1$. The condition on $d_2(\delta)$ implies $d_1\neq 0$. The condition on $d_1(\delta)$ implies that $(d_1,e_1,f_1)\neq(0,0,0)$. The condition on $d_2(\delta)$ implies $p\mid d_1f_1$ such that $p\mid d_1$ or $p\mid f_1$, namely $d_1=0$ or $f_1=0$. There are $(p-1)(2p+1)$ such matrices of type $1$.
\par
Let us consider the matrices of type $2$. The condition on $d_2(\delta)$ implies $p\mid f_2$. There are $p^4$ such matrices of type $2$.
\par
Let us consider the matrices of type $3$. The condition on $d_2(\delta)$ implies $f_1=0$. There are $p^3$ such matrices of type $3$.
\par
Let us consider the matrices of type $4$. The condition on $d_2(\delta)$ implies $p\mid e_2$. There are $p^3$ such matrices of type $4$.
\par
Let us consider the matrices of type $5$. The condition on $d_2(\delta)$ implies $e_1=0$. There are $p$ such matrices of type $5$.
\par
Let us consider the matrices of type $6$. The condition on $d_2(\delta)$ implies $p\mid d_2$. There are $p$ such matrices of type $6$.
\par
Let us consider the matrices of type $7$. The condition on $d_2(\delta)$ implies $d_1=0$. There is $1$ such matrix of type $7$.
\par
One can recover the decomposition in $\Lambda$-right cosets given in \eqref{eq_dec_1_p_p^2} and the value of the degree given in \eqref{eq_deg_1_p_p^2} by summing all the contributions in the previous paragraphs.
\end{proof}
\subsection{Degree of $\Lambda\textnormal{diag}(1,p^2,p^4)\Lambda$}%
\begin{proposition}\label{propo_dec_1_p^2_p^4}
One has
\begin{equation}\label{eq_deg_1_p^2_p^4}
\textnormal{deg}\left(\textnormal{diag}(1,p^2,p^4)\right)=p^5(p+1)(p^2+p+1).
\end{equation}
\end{proposition}
\begin{proof}[\proofname{} of Proposition \ref{propo_dec_1_p^2_p^4}]%
By \eqref{eq_practice}, the possible upper-triangular column reduced matrices $\delta$ that can occur in the decomposition into $\Lambda$-right cosets are
\begin{eqnarray*}
\text{Type $1$:} \begin{pmatrix}
p^4 & d_2 &  \\
& p^2 &  \\
& & 1
\end{pmatrix} & \rightsquigarrow & \uple{d}(\delta)=(1,(p^2,d_2),p^6), \\
\text{Type $2$:} \begin{pmatrix}
p^4 &  & e_2 \\
& 1 & f_2 \\
& & p^2
\end{pmatrix} & \rightsquigarrow & \uple{d}(\delta)=(1,(p^2,e_2),p^6), \\
\text{Type $3$:} \begin{pmatrix}
p^2 & d_4 &  \\
& p^4 &  \\
& & 1
\end{pmatrix} & \rightsquigarrow & \uple{d}(\delta)=(1,(p^2,d_4),p^6), \\
\text{Type $4$:} \begin{pmatrix}
p^2 &  & e_4 \\
& 1 & f_4 \\
& & p^4
\end{pmatrix} & \rightsquigarrow & \uple{d}(\delta)=(1,(p^2,e_4),p^6), \\
\text{Type $5$:} \begin{pmatrix}
1 & d_4 & e_2 \\
& p^4 & f_2 \\
& & p^2
\end{pmatrix} & \rightsquigarrow & \uple{d}(\delta)=(1,(p^2,f_2,d_4f_2),p^6), \\
\text{Type $6$:} \begin{pmatrix}
1 & d_2 & e_4 \\
& p^2 & f_4 \\
& & p^4
\end{pmatrix} & \rightsquigarrow & \uple{d}(\delta)=(1,(p^2,f_4,d_2f_4),p^6)
\end{eqnarray*}
and
\begin{eqnarray*}
\text{Type $7$:} \begin{pmatrix}
p^4 & d_1 & e_1 \\
& p & f_1 \\
& & p
\end{pmatrix} & \rightsquigarrow & \uple{d}(\delta)=((p,d_1,e_1,f_1),(p^2,pd_1,d_1f_1-pe_1),p^6), \\
\text{Type $8$:} \begin{pmatrix}
p & d_4 & e_1 \\
& p^4 & f_1 \\
& & p
\end{pmatrix} & \rightsquigarrow & \uple{d}(\delta)=((p,d_4,e_1,f_1),(p^2,pf_1,pd_4,d_4f_1),p^6), \\
\text{Type $9$:} \begin{pmatrix}
p & d_1 & e_4 \\
& p & f_4 \\
& & p^4
\end{pmatrix} & \rightsquigarrow & \uple{d}(\delta)=((p,d_1,e_4,f_4),(p^2,pf_4,d_1f_4-pe_4),p^6)
\end{eqnarray*}
and
\begin{eqnarray*}
\text{Type $10$:} \begin{pmatrix}
p^3 & d_3 &  \\
& p^3 &  \\
& & 1
\end{pmatrix} & \rightsquigarrow & \uple{d}(\delta)=(1,(p^3,d_3),p^6), \\
\text{Type $11$:} \begin{pmatrix}
p^3 &  & e_3 \\
& 1 & f_3 \\
& & p^3
\end{pmatrix} & \rightsquigarrow & \uple{d}(\delta)=(1,(p^3,e_3),p^6), \\
\text{Type $12$:} \begin{pmatrix}
1 & d_3 & e_3 \\
& p^3 & f_3 \\
& & p^3
\end{pmatrix} & \rightsquigarrow & \uple{d}(\delta)=(1,(p^3,f_3,d_3f_3),p^6)
\end{eqnarray*}
and
\begin{eqnarray*}
\text{Type $13$:} \begin{pmatrix}
p^3 & d_2 & e_1 \\
& p^2 & f_1 \\
& & p
\end{pmatrix} & \rightsquigarrow & \uple{d}(\delta)=((p,d_2,e_1,f_1),(p^3,pd_2,d_2f_1-p^2e_1),p^6) \\
\text{Type $14$:} \begin{pmatrix}
p^3 & d_1 & e_2 \\
& p & f_2 \\
& & p^2
\end{pmatrix} & \rightsquigarrow & \uple{d}(\delta)=((p,d_1,e_2,f_2),(p^3,p^2d_1,d_1f_2-pe_2),p^6) \\
\text{Type $15$:} \begin{pmatrix}
p^2 & d_3 & e_1 \\
& p^3 & f_1 \\
& & p
\end{pmatrix} & \rightsquigarrow & \uple{d}(\delta)=((p,d_3,e_1,f_1),(p^3,pd_3,p^2f_1,d_3f_1),p^6) \\
\text{Type $16$:} \begin{pmatrix}
p^2 & d_1 & e_3 \\
& p & f_3 \\
& & p^3
\end{pmatrix} & \rightsquigarrow & \uple{d}(\delta)=((p,d_1,e_3,f_3),(p^3,d_1f_3-pe_3,p^2f_3),p^6)
\end{eqnarray*}
and
\begin{eqnarray*}
\text{Type $17$:} \begin{pmatrix}
p & d_3 & e_2 \\
& p^3 & f_2 \\
& & p^2
\end{pmatrix} & \rightsquigarrow & \uple{d}(\delta)=((p,d_3,e_2,f_2),(p^3,pf_2,p^2d_3,d_3f_2),p^6) \\
\text{Type $18$:} \begin{pmatrix}
p & d_2 & e_3 \\
& p^2 & f_3 \\
& & p^3
\end{pmatrix} & \rightsquigarrow & \uple{d}(\delta)=((p,d_2,e_3,f_3),(p^3,pf_3,d_2f_3-p^2e_3),p^6)
\end{eqnarray*}
and
\begin{eqnarray*}
\text{Type $19$:} \begin{pmatrix}
p^2 & d_2 & e_2 \\
& p^2 & f_2 \\
& & p^2
\end{pmatrix} & \rightsquigarrow & \uple{d}(\delta)=((p^2,d_2,e_2,f_2),(p^4,p^2d_2,p^2f_2,d_2f_2-p^2e_2),p^6)
\end{eqnarray*}
where $0\leq d_j, e_j, f_j<p^j$ for $j=1,2,3,4$. Let us count the matrices among the previous ones, whose determinantal vector is the same as the one of $\textnormal{diag}(1,p^2,p^4)$, namely $(1,p^2,p^6)$.
\par
Let us consider the matrices of type $1$. The condition on $d_2(\delta)$ implies $d_2=0$. There is $1$ relevant matrix of type $1$.
\par
Let us consider the matrices of type $2$. The condition on $d_2(\delta)$ implies $e_2=0$. There are $p^2$ relevant matrices of type $2$.
\par
Let us consider the matrices of type $3$. The condition on $d_2(\delta)$ implies $p^2\mid d_4$. There are $p^2$ relevant matrices of type $3$.
\par
Let us consider the matrices of type $4$. The condition on $d_2(\delta)$ implies $p^2\mid e_4$. There are $p^6$ relevant matrices of type $4$.
\par
Let us consider the matrices of type $5$. The condition on $d_2(\delta)$ implies $e_2=0$ . There are $p^6$ relevant matrices of type $5$.
\par
Let us consider the matrices of type $6$. The condition on $d_2(\delta)$ implies $p^2\mid f_4$ . There are $p^8$ relevant matrices of type $6$.
\par
Let us consider the matrices of type $7$. The condition on $d_2(\delta)$ implies $d_1=e_1=0$ and the condition on $d_1(\delta)$ implies $f_1\neq 0$. There are $p-1$ relevant matrices of type $7$.
\par
Let us consider the matrices of type $8$. The condition on $d_2(\delta)$ implies $f_1=0$ and $p\mid d_4$. The condition on $d_1(\delta)$ implies $e_1\neq 0$. There are $p^3(p-1)$ relevant matrices of type $8$.
\par
Let us consider the matrices of type $9$. The condition on $d_2(\delta)$ implies $p\mid f_4$ and $p\mid d_1f_4/p-e_4$. One has $d_1\neq 0$ since otherwise $d_1(\delta)=1=(p,e_4)$ and $d_2(\delta)=p(p,e_4)=p\neq p^2$. Thus, $d_1$ is invertible modulo $p$ and $f_4/p\equiv e_4\overline{d_1}\pmod{p}$ such that $f_4/p$ can take $p^2$ values. There are $(p-1)p^6$ relevant matrices of type $9$.
\par
Let us consider the matrices of type $10$. The condition on $d_2(\delta)$ implies $p^2\mid\mid d_3$. There are $p-1$ relevant matrices of type $10$.
\par
Let us consider the matrices of type $11$. The condition on $d_2(\delta)$ implies $p^2\mid\mid e_3$. There are $(p-1)p^3$ relevant matrices of type $11$.
\par
Let us consider the matrices of type $12$. The condition on $d_2(\delta)$ implies $p^2\mid\mid f_3$. There are $(p-1)p^6$ relevant matrices of type $12$.
\par
Let us consider the matrices of type $13$. Note that $(e_1,f_1)\neq (0,0)$ since otherwise $d_1(\delta)=1=(p,d_2)$, which implies that $d_2(\delta)=(pd_2,p^3)=p\neq p^2$. As a consequence, $d_1(\delta)=1=(p,d_2,e_1,f_1)$. The fact that $d_2(\delta)=p^2$ implies that $p\mid d_2$ and $p\mid f_1d_2/p$, namely $f_1=0$ or $d_2=0$. If $d_2=0$ then $d_2(\delta)=p^2=(p^3,p^2e_1)$ such that $e_1\neq 0$. There are $p(p-1)$ such matrices. If $d_2\neq 0$ then $f_1=0$, $d_2(\delta)=p^2(p,d_2/p,e_1)=p^2$ since $d_2/p$ is coprime with $p$ and $d_1(\delta)=1=(p,e_1)$ such that $e_1\neq 0$. There are $(p-1)^2$ such matrices. Finally, there are $(p-1)(2p-1)$ relevant matrices of type $13$.
\par
Let us consider the matrices of type $14$. The fact that $d_2(\delta)=p^2$ implies that  $p^2\mid d_1f_2-pe_2$. If $d_1=0$ then $p\mid e_2$ and $d_2(\delta)=p^2=(p^3,p^2e_2/p)$ if $e_2\neq 0$. $d_1(\delta)=1=(p,f_2)$ implies that $p\nmid f_2$. There are $(p-1)(p^2-p)$ such matrices. If $d_1\neq 0$ then the value of $f_2$ is fixed by $f_2\equiv pe_2\overline{d_1}\pmod{p^2}$ and $d_1(\delta)=(p,d_1)=1$. There are $p^2(p-1)$ such matrices. Finally, there are $(p-1)(2p^2-p)$ relevant matrices of type $14$.
\par
Let us consider the matrices of type $15$. The condition $d_2(\delta)=p^2$ implies that $p\mid d_3$ and $p\mid f_1d_3/p$. If $f_1=0$ then $d_2(\delta)=p^2=p^2(p,d_3/p)$ such that $p\mid\mid d_3$. The condition $d_1(\delta)=1=(p,e_1)$ implies that $e_1\neq 0$. There are $(p^2-p)(p-1)$ such matrices. If $f_1\neq 0$ then $p^2\mid d_3$ and $d_1(\delta)=1$. There are $p^2(p-1)$ such matrices. Finally, there are $(p-1)(2p^2-p)$ relevant matrices of type $15$.
\par
Let us consider the matrices of type $16$. The condition $d_2(\delta)=p^2$ implies that $p^2\mid d_1f_3-pe_3$. If $p\mid e_3$ then $p^2\mid d_1f_3$. If $p\mid e_3$ and $p\mid d_1$ then $d_1=0$ and the condition $d_1(\delta)=1=(p,f_3)$ implies that $p\nmid f_3$ and $d_2(\delta)=p^2$. There are $p^2(p^3-p^2)$ such matrices. If $p\mid e_3$ and $p\nmid d_1$ then $p^2\mid f_3$ then $d_2(\delta)=p^2(p,d_1f_3/p^2-e_3/p)\neq p^2$ if and only if $f_3/p^2\equiv \overline{d_1}e_3/p\pmod{p}$, which given $d_1$ and $e_3/p$ can happen for only one value of $f_3/p^2$. There are $(p-1)p^2(p-1)$ such matrices. If $p\nmid e_3$ then $d_1(\delta)=1=(p,d_1,e_3,f_3)$. The condition $p^2\mid d_1f_3-pe_3$ implies that $p^2\nmid d_1f_3$ and $p\nmid d_1$ but $p\mid f_3$. The condition $d_2(\delta)$ implies that $p\mid\mid d_1f_3/p-e_3$. Given $d_1$ and $e_3$, there are $p$ choices for $f_3/p$ given by $f_3/p\equiv\overline{d_1}e_3\pmod{p}$ but one has to remove the value satisfying $f_3/p\equiv\overline{d_1}e_3\pmod{p^2}$. There are $(p-1)(p^3-p^2)(p-1)$ such matrices. Finally, there are $p^3(p-1)(2p-1)$ relevant matrices of type $16$.
\par
Let us consider the matrices of type $17$. The condition $d_2(\delta)=p^2$ implies that $p\mid f_2$. If $f_2=0$ then $d_2(\delta)=p^2=p^2(p,d_3)$ such that $p\nmid d_3$, which implies $d_1(\delta)=1$. There are $(p^3-p^2)p^2$ such matrices. If $f_2\neq 0$ then $p\mid d_3$ since $p\mid d_3f_2/p$, in which case $d_2(\delta)=p^2$. The condition $d_1(\delta)=1$ implies that $p\nmid e_2$. There are $p^2(p^2-p)(p-1)$ such matrices. Finally, there are $p^3(p-1)(2p-1)$ relevant matrices of type $17$.
\par
Let us consider the matrices of type $18$. The condition $d_2(\delta)=p^2$ implies that $p\mid f_3$ and $p\mid d_2f_3/p$. If $p^2\mid f_3$ then $d_2(\delta)=p^2=p^2(p,d_2f_3/p^2-e_3)$. One has to remove the $p^2$ values of $e_3$ satisfying $e_3\equiv d_2f_3/p^2\pmod{p}$. In this case, one has $d_1(\delta)=(p,d_2,e_3)=1$ since if $p\mid(d_2,e_3)$ then $(p,d_2f_3/p^2-e_3)\neq 1$. There are $p^2(p^3-p^2)p$ such matrices. If $p^2\nmid f_3$ then $p\mid d_2$ and the conditions on $d_1(\delta)$ and $d_2(\delta)$ are satisfied. There are $p(p^3-p^2)(p^2-p)$ such matrices. Finally, there are $p^4(p-1)(2p-1)$ relevant matrices of type $18$.
\par
Let us consider the matrices of type $19$. The condition on $d_2(\delta)$ implies that $p^2\mid d_2f_2$. If $d_2=0$ then $d_2(\delta)=p^2=p^2(p^4,e_2,f_2)$ and $d_1(\delta)=1=(p^2,e_2,f_2)$. One has to remove the couples $(e_2,f_2)$ satisfying $p\mid e_2$ and $p\mid f_2$, namely $p^2$ couples. There are $p^4-p^2$ such matrices. If $d_2\neq 0$ and $f_2=0$ then $d_2(\delta)=p^2=p^2(p^2,d_2,e_2)$ and $d_1(\delta)=1=(p^2,d_2,e_2)$. One has to remove the couples $(d_2,e_2)$ satisfying $p\mid d_2$ and $p\mid e_2$, namely $(p-1)p$ couples. There are $(p^2-1)p^2-(p-1)p$ such matrices. If $d_2\neq 0$ and $f_2\neq 0$ then $d_2(\delta)=p^2=p^2(p^2,pd_2/p,pf_2/p,d_2f_2/p^2-e_2)$ and $d_1(\delta)=1=(p^2,e_2)$. Thus, $p\nmid e_2$ and $p\nmid d_2f_2/p^2-e_2$. Among the $p^2-p$ values of $e_2$ satisfying $p\nmid e_2$, one has to remove these satisfying $e_2\equiv d_2f_2/p^2\pmod{p}$ of cardinal $p$. There are $(p-1)(p^2-2p)(p-1)$ such matrices. Finally, there are $p(p-1)(3p^2-p+1)$ relevant matrices of type $19$. 
\par
One can recover the value of the degree given in \eqref{eq_deg_1_p^2_p^4} by summing all the contributions in the previous paragraphs.
\end{proof}
\subsection{Degree of $\Lambda\textnormal{diag}(1,p^3,p^3)\Lambda$}%
\begin{proposition}\label{propo_dec_1_p^3_p^3}
One has
\begin{equation}\label{eq_deg_1_p^3_p^3}
\textnormal{deg}\left(\textnormal{diag}(1,p^3,p^3)\right)=p^4(p^2+p+1).
\end{equation}
\end{proposition}
\begin{proof}[\proofname{} of Proposition \ref{propo_dec_1_p^3_p^3}]%
By \eqref{eq_practice}, the possible upper-triangular column reduced matrices $\delta$ that can occur in the decomposition into $\Lambda$-right cosets are
\begin{eqnarray*}
\text{Type $1$:} \begin{pmatrix}
p^3 & d_3 &  \\
& p^3 &  \\
& & 1
\end{pmatrix} & \rightsquigarrow & \uple{d}(\delta)=(1,(p^3,d_3),p^6), \\
\text{Type $2$:} \begin{pmatrix}
p^3 &  & e_3 \\
& 1 & f_3 \\
& & p^3
\end{pmatrix} & \rightsquigarrow & \uple{d}(\delta)=(1,(p^3,e_3),p^6), \\
\text{Type $3$:} \begin{pmatrix}
1 & d_3 & e_3 \\
& p^3 & f_3 \\
& & p^3
\end{pmatrix} & \rightsquigarrow & \uple{d}(\delta)=(1,(p^3,f_3,d_3f_3),p^6)
\end{eqnarray*}
and\begin{eqnarray*}
\text{Type $4$:} \begin{pmatrix}
p^3 & d_2 & e_1 \\
& p^2 & f_1 \\
& & p
\end{pmatrix} & \rightsquigarrow & \uple{d}(\delta)=((p,d_2,e_1,f_1),(p^3,pd_2,d_2f_1-p^2e_1),p^6), \\
\text{Type $5$:} \begin{pmatrix}
p^3 & d_1 & e_2 \\
& p & f_2 \\
& & p^2
\end{pmatrix} & \rightsquigarrow & \uple{d}(\delta)=((p,d_1,e_2,f_2),(p^3,p^2d_1,d_1f_2-pe_2),p^6), \\
\text{Type $6$:} \begin{pmatrix}
p^2 & d_3 & e_1 \\
& p^3 & f_1 \\
& & p
\end{pmatrix} & \rightsquigarrow & \uple{d}(\delta)=((p,d_3,e_1,f_1),(p^3,pd_3,p^2f_1,d_3f_1),p^6), \\
\text{Type $7$:} \begin{pmatrix}
p^2 & d_1 & e_3 \\
& p & f_3 \\
& & p^3
\end{pmatrix} & \rightsquigarrow & \uple{d}(\delta)=((p,d_1,e_3,f_3),(p^3,p^2f_1,d_1f_3-pe_3),p^6), \\
\text{Type $8$:} \begin{pmatrix}
p & d_3 & e_2 \\
& p^3 & f_2 \\
& & p^2
\end{pmatrix} & \rightsquigarrow & \uple{d}(\delta)=((p,d_3,e_2,f_2),(p^3,p^2d_3,pf_2,d_3f_2),p^6), \\
\text{Type $9$:} \begin{pmatrix}
p & d_1 & e_3 \\
& p^2 & f_3 \\
& & p^3
\end{pmatrix} & \rightsquigarrow & \uple{d}(\delta)=((p,d_1,e_4,f_4),(p^3,pf_3,d_2f_3-p^2e_3),p^6)
\end{eqnarray*}
and
\begin{eqnarray*}
\text{Type $10$:} \begin{pmatrix}
p^2 & d_2 & e_2 \\
& p^2 & f_2 \\
& & p^2
\end{pmatrix} & \rightsquigarrow & \uple{d}(\delta)=((p^2,d_2,e_2,f_2),(p^4,p^2d_2,p^2f_2,d_2f_2-p^2e_2),p^6)
\end{eqnarray*}
where $0\leq d_j, e_j, f_j<p^j$ for $j=1,2,3$. Let us count the matrices among the previous ones, whose determinantal vector is the same as the one of $\textnormal{diag}(1,p^3,p^3)$, namely $(1,p^3,p^6)$.
\par
Let us consider the matrices of type $1$. The condition on $d_2(\delta)$ implies $d_3=0$. There is $1$ relevant matrix of type $1$.
\par
Let us consider the matrices of type $2$. The condition on $d_2(\delta)$ implies $e_3=0$. There are $p^3$ relevant matrices of type $2$.
\par
Let us consider the matrices of type $3$. The condition on $d_2(\delta)$ implies $f_3=0$. There are $p^6$ relevant matrices of type $3$.
\par
Let us consider the matrices of type $4$. The condition on $d_2(\delta)$ implies $d_2=0$ and $e_1=0$. Then, $d_1(\delta)=1=(p,f_1)$ such that $f_1\neq 0$. There are $p-1$ relevant matrices of type $4$.
\par
Let us consider the matrices of type $5$. The condition on $d_2(\delta)$ implies $d_1=0$ and $e_2=0$. Then, $d_1(\delta)=1=(p,f_2)$ such that $p\nmid f_2$. There are $p^2-p$ relevant matrices of type $5$.
\par
Let us consider the matrices of type $6$. The condition on $d_2(\delta)$ implies $f_1=0$ and $p^2\mid d_3$. Then, $d_1(\delta)=1=(p,e_1)$ such that $e_1\neq 0$. There are $p(p-1)$ relevant matrices of type $6$.
\par
Let us consider the matrices of type $7$. The condition on $d_2(\delta)$ implies $p\mid f_3$ and $p\mid d_1f_3/p-e_3$. One has $d_1\neq 0$ since otherwise $p^2\mid e_2$ by the condition on $d_2(\delta)$ such that $d_1(\delta)=p\neq 1$. Thus, $d_1$ is invertible modulo $p$ and $f_3/p\equiv e_3\overline{d_1}\pmod{p^2}$ is fixed. There are $(p-1)p^3$ relevant matrices of type $7$.
\par
Let us consider the matrices of type $8$. The condition on $d_2(\delta)$ implies $p\mid d_3$ and $f_2=0$. Then, $d_1(\delta)=1=(p,e_2)$ such that $p\nmid e_2$. There are $p^2(p^2-p)$ relevant matrices of type $8$.
\par
Let us consider the matrices of type $9$. The condition on $d_2(\delta)$ implies $p^2\mid f_3$ and $p\mid d_2f_3/p-e_3$. If $f_3=0$ then $p\mid e_3$ and $d_1(\delta)=1=(p,d_2)$ such that $p\nmid d_2$. There are $(p^2-p)p^2$ such matrices. If $f_3\neq 0$ then $d_2\equiv e_3\overline{f_3/p^2}\pmod{p}$ can take $p$ values. Then, $d_1(\delta)=1=(p,e_3)$ such that $p\nmid e_3$. There are $p(p^3-p^2)(p-1)$ such matrices. Finally, there are $p^4(p-1)$ relevant matrices of type $9$.
\par
Let us consider the matrices of type $10$. The condition on $d_2(\delta)$ implies $p\mid d_2$, $p\mid f_2$ and $p\mid d_2f_2/p^2-e_2$. One has $d_2\neq 0$ since otherwise $d_2(\delta)=p^2(p^2,e_2)=p^2d_1(\delta)=p^2\neq p^3$. Thus, $d_2/p$ is invertible modulo $p$ and $f_2$ is fixed by $f_2/p\equiv e_2\overline{d_2/p}\pmod{p}$. Then, $d_1(\delta)=1=(p,e_2)$ such that $p\nmid e_2$, $p\nmid f_2/p$ and $d_2(\delta)=p^3$. There are $(p-1)(p^2-p)$ relevant matrices of type $10$.
\par
One can recover the value of the degree given in \eqref{eq_deg_1_p^3_p^3} by summing all the contributions in the previous paragraphs.
\end{proof}
\subsection{Degree of $\Lambda\textnormal{diag}(1,1,p^3)\Lambda$}%
\begin{proposition}\label{propo_dec_p_p_p^4}
One has
\begin{equation}\label{eq_deg_p_p_p^4}
\textnormal{deg}\left(\textnormal{diag}(1,1,p^3)\right)=p^4(p^2+p+1).
\end{equation}
\end{proposition}
\begin{proof}[\proofname{} of Proposition \ref{propo_dec_p_p_p^4}]%
By \eqref{eq_practice}, the possible upper-triangular column reduced matrices $\delta$ that can occur in the decomposition of the $\Lambda$-double coset $\Lambda\textnormal{diag}(1,1,p^3)\Lambda$ into $\Lambda$-right cosets are
\begin{eqnarray*}
\text{Type $1$:} \begin{pmatrix}
p^3 &  &  \\
& 1 &  \\
& & 1
\end{pmatrix} & \rightsquigarrow & \uple{d}(\delta)=(1,1,p^3), \\
\text{Type $2$:} \begin{pmatrix}
1 & d_3 &  \\
& p^3 &  \\
& & 1
\end{pmatrix} & \rightsquigarrow & \uple{d}(\delta)=(1,1,p^3), \\
\text{Type $3$:} \begin{pmatrix}
1 &  & e_3 \\
& 1 & f_3 \\
& & p^3
\end{pmatrix} & \rightsquigarrow & \uple{d}(\delta)=(1,1,p^3)
\end{eqnarray*}
and
\begin{eqnarray*}
\text{Type $4$:} \begin{pmatrix}
1 & d_1 & e_2 \\
& p & f_2 \\
& & p^2
\end{pmatrix} & \rightsquigarrow & \uple{d}(\delta)=(1,(p,f_2),p^3), \\
\text{Type $5$:} \begin{pmatrix}
1 & d_2 & e_1 \\
& p^2 & f_1 \\
& & p
\end{pmatrix} & \rightsquigarrow & \uple{d}(\delta)=(1,(p,f_1),p^3), \\
\text{Type $6$:} \begin{pmatrix}
p &  & e_2 \\
& 1 & f_2 \\
& & p^2
\end{pmatrix} & \rightsquigarrow & \uple{d}(\delta)=(1,(p,e_2),p^3), \\
\text{Type $7$:} \begin{pmatrix}
p^2 &  & e_1 \\
& 1 & f_1 \\
& & p
\end{pmatrix} & \rightsquigarrow & \uple{d}(\delta)=(1,(p,e_1),p^3), \\
\text{Type $8$:} \begin{pmatrix}
p & d_2 &  \\
& p^2 &  \\
& & 1
\end{pmatrix} & \rightsquigarrow & \uple{d}(\delta)=(1,(p,d_2),p^3), \\
\text{Type $9$:} \begin{pmatrix}
p^2 & d_1 &  \\
& p &  \\
& & 1
\end{pmatrix} & \rightsquigarrow & \uple{d}(\delta)=(1,(p,d_1),p^3)
\end{eqnarray*}
and
\begin{eqnarray*}
\text{Type $10$:} \begin{pmatrix}
p & d_1 & e_1 \\
& p & f_1 \\
& & p
\end{pmatrix} & \rightsquigarrow & \uple{d}(\delta)=((p,d_1,e_1,f_1),(p^2,pd_1,pf_1,d_1f_1-pe_1),p^3)
\end{eqnarray*}
where $0\leq d_j, e_j, f_j<p^j$ for $j=1,2,3$. Let us count the matrices among the previous ones, whose determinantal vector is the same as the one of $\textnormal{diag}(1,1,p^3)$, namely $(1,1,p^3)$.
\par
Let us consider the matrices of type $1$. There is $1$ relevant matrix of type $1$.
\par
Let us consider the matrices of type $2$. There are $p^3$ relevant matrices of type $2$.
\par
Let us consider the matrices of type $3$. There are $p^6$ relevant matrices of type $3$.
\par
Let us consider the matrices of type $4$. The condition on $d_2(\delta)$ implies $p\nmid f_2$. There are $p^3(p^2-p)$ relevant matrices of type $4$.
\par
Let us consider the matrices of type $5$. The condition on $d_2(\delta)$ implies $f_1\neq 0$. There are $p^3(p-1)$ relevant matrices of type $5$.
\par
Let us consider the matrices of type $6$. The condition on $d_2(\delta)$ implies $p\nmid e_2$. There are $p^2(p^2-p)$ relevant matrices of type $6$.
\par
Let us consider the matrices of type $7$. The condition on $d_2(\delta)$ implies $e_1\neq 0$. There are $p(p-1)$ relevant matrices of type $6$.
\par
Let us consider the matrices of type $8$. The condition on $d_2(\delta)$ implies $p\nmid d_2$. There are $p^2-p$ relevant matrices of type $7$.
\par
Let us consider the matrices of type $9$. The condition on $d_2(\delta)$ implies $d_1\neq 0$. There are $p-1$ relevant matrices of type $9$.
\par
Let us consider the matrices of type $10$. One has $d_1\neq 0$ since otherwise $p\mid d_2(\delta)$. Thus, $d_1(\delta)=1$. In addition, $f_1\neq 0$ since otherwise $p\mid d_2(\delta)$. There are $p(p-1)^2$ relevant matrices of type $10$.
\par
One can recover the value of the degree given in \eqref{eq_deg_p_p_p^4} by summing all the contributions in the previous paragraphs.
\end{proof}
\bibliographystyle{alpha}
\bibliography{biblio}
\end{document}